\documentclass{article}
\usepackage[utf8]{inputenc}

\usepackage{amssymb}
\usepackage{latexsym}
\usepackage{amsfonts}
\usepackage{amsthm}
\usepackage{amsmath}
\usepackage{epigraph}
\usepackage{perpage}
\usepackage[dvipsnames]{xcolor}
\MakePerPage{footnote}

 \usepackage{hyperref}
 
 \hypersetup{colorlinks = true, linkcolor = black, citecolor= gray, urlcolor= MidnightBlue}
  
  \usepackage{color}
  \usepackage{xcolor}
  \usepackage{mathtools}
  \usepackage{tikz}
  \usepackage{float}
  \usepackage{bbm}
  \usepackage{listings}
  \usepackage{enumerate}
  \usepackage{subcaption}
	\usepackage{tikz-cd}
	\usepackage{mathtools}
  \usepackage{todonotes}
  \usepackage{adjustbox}
  \usepackage{rotating}
  \usepackage{environ}
  \usepackage{mathbbol}
  \usepackage{appendix}
  \usepackage{tabularx}
  \usepackage{breqn}
  
  \usepackage[english]{babel}
  \usepackage{biblatex}
  
  \newcommand{\mi}{\mathrm{i}}

\usepackage[scr=boondoxo,scrscaled=1.05]{mathalfa}
	
\usetikzlibrary{shapes}

	\usetikzlibrary{calc}

\usepackage{pgf}

\usepackage{tikz}
\usetikzlibrary{arrows.meta}
\usetikzlibrary{arrows}
\usetikzlibrary{calc}
\usetikzlibrary{patterns}
\usetikzlibrary{shapes.misc}
\pgfdeclarelayer{bg}
\pgfsetlayers{bg,main}

\usetikzlibrary{lindenmayersystems}

  \newtheorem{theorem}{Theorem}[section]
  \newtheorem{lemma}[theorem]{Lemma}
  \newtheorem{prop}[theorem]{Proposition}
  \newtheorem{cor}[theorem]{Corollary}

  \theoremstyle{definition}
  
  \newtheorem{defn}[theorem]{Definition}
  \newtheorem{ex}[theorem]{Example}

  \newtheorem{remark}[theorem]{Remark}

  \DeclareMathOperator*{\argmin}{arg\,min}
  
  \def\l{\left\langle}
  \def\r{\right\rangle}
  \def\1{\mathbbm{1}  }
  
  \def\R{\mathbb{R}}

  \def\bom{ \boldsymbol\omega}

  \def\N{\mathbb{N}}
  \def\Z{\mathbb{Z}}

  \def\ellin{ \prescript{}{\infty}{\ell}_{\infty}    }

\def\bx{\mathbf{x}}

\newcommand{\pfrac}[2]{\frac{\partial #1}{\partial #2}}

  \usepackage{array}
  
  \newcolumntype{P}[1]{>{\centering\arraybackslash}p{#1}}

\title{Provably Stable Full-Spectrum Dispersion Relation Preserving Schemes}
\author{Christopher Williams\footnote{corresponding author} \ Kenneth Duru  }
\date{}

\begin{document}

\maketitle

\begin{abstract}
    The dispersion error is often the dominant error for computed solutions of wave propagation problems with high-frequency components. In this paper, we define and give explicit examples of $\alpha$-dispersion-relation-preserving schemes. These are dual-pair finite-difference schemes for systems of hyperbolic partial differential equations which preserve the dispersion-relation of the continuous problem uniformly to an $\alpha \%$-error tolerance. We give a general framework to design provably stable finite difference operators that preserve the dispersion relation for hyperbolic systems such as the elastic wave equation. The operators we derive here can resolve the highest frequency ($\pi$-mode) present on any equidistant grid at a tolerance of $5\%$ error. This significantly improves on the current standard that have a tolerance of $100 \%$ error.
\end{abstract}

\section{Introduction}

It is well known that the dispersion error of a computational scheme has a great impact on the accuracy of a numerical simulation for hyperbolic partial differential equations (PDEs).  Various efforts have been made to minimise the dispersion error for computational schemes \cite{tam1993dispersion,linders2015uniformly}, but all of these efforts have been unable to resolve high-frequency wave modes, such as the highest-frequency mode on a computational grid called the $\pi$-mode. 
These attempts have also exclusively used central finite difference (FD) stencils on co-located grids, as this has generally been accepted as necessary for a stable numerical scheme. Recently, the dual-pairing summation-by-parts (SBP) FD framework \cite{Mattsson2017,DovgilovichSofronov2015} has shown that this is not necessarily true. In some applications, staggered FD methods \cite{Yee1968,Virieux1986} on  Cartesian meshes can minimise  numerical dispersion errors. 
However, the design of high order accurate and stable staggered FD methods  in complex geometries is a challenge. See some progress  in this direction \cite{OssianAnders2020,Ossian2021}.

In this paper, we exploit the dual-pairing SBP FD framework on co-located grids to design FD schemes which can resolve high frequencies to an acceptable error tolerance for use in application. The dual-pairing SBP FD framework introduces additional degrees of freedom that can be tuned to diminish numerical dispersion errors.
In doing so we minimise two major numerical issues currently present in computational wave simulation, these are: the presence of spurious high-frequency wave modes; and, the reduced grid-refinement needed for the same error bounds by a factor of more than two in one-spatial dimension. 
The result of this improvement is the absence of computationally fatal spurious wave modes in numerically computed solutions, and an efficiency increase that is exponential with the dimension of the problem. 
For instance, in three-dimensions for a time-dependent problem our schemes achieve the same error bounds as traditional methods with $\approx 30$ times less computational effort. 

The structure of this paper is as follows. 
First we recall the traditional SBP \cite{kreiss1974finite,strand1994summation} and dual-pairing SBP \cite{Mattsson2017,DovgilovichSofronov2015} framework. 
Here we identify what is necessary for our scheme to be provably stable. 
Next we define what an $\alpha$-dispersion-relation-preserving scheme is within this framework.
We then turn to the construction of a $\alpha$-DRP scheme, giving a methodology for constructing upwind dual-pairing schemes that are able to resolve high-frequencies to $\alpha \%$ error. 
We find that this problem can be reduced to two optimisation problems. 
The first optimisation pertains to the design of an interior stencil that is a non-convex smooth optimisation problem with a convex relaxation. 
The second relates to the boundary stencils needed, and can be posed as a non-smooth convex minimisation problem. 
We then compute explicit examples and present the numerical error properties of these schemes. 
Small test examples are presented and compared to standard methods. 

\section{The Dual-Pairing SBP FD Framework}
Previous work on FD methods has primarily focused on central stencil schemes, with carefully chosen boundary stencils so that the SBP property holds. It was thought that central FD stencils was necessary for provably stable schemes for general systems.
Recently, progression on skewed stencil schemes in the SBP framework has shown that non-central stencils may also be stable. 
In this section we review what is necessary for a skewed stencil scheme to maintain stability. 

\subsection{Notations}

The following definitions are required for our following study. 
We will always use the variable $x$ to be the spatial variable and work with the closed interval $x \in [-\pi, \pi]$ of length $2\pi$.  Let $u, v \in L^2([-\pi, \pi])$ with the standard inner product 
\begin{align}
    \l u,  v \r = \int_{-\pi}^{\pi} uv  . 
\end{align}
Let $n \in \N$ be the number of grid points on a uniformly spaced grid,
\begin{align}\label{eq:grid_points}
    x_i \coloneqq -\pi + {(i-1)}h,  \qquad h = \frac{2\pi}{n-1}, \qquad i = 1, 2, \cdots n.
\end{align}
Here $ x_i$ are the grid-points and $h>0$ is the spatial step.
For a given $n >0 $ we define the grid vector $\bx \in \R^n$ as
\begin{align}
    \bx \coloneqq (x_1, \dots, x_n)^T \in \R^n.
\end{align}
For a function $u \in L^2([-\pi, \pi])$ we use the bold character $\mathbf{u}$ to denote the grid function, that is the restriction of $u$ on the grid,
\begin{align}
    \mathbf{u} \coloneqq (u(x_1), \dots, u(x_n))^T \in \R^n. 
\end{align}
We introduce the diagonal and positive definite matrix 
\begin{align}
    H \coloneqq \mathrm{diag}( (h_1, \dots, h_n) ) \in \R^{n \times n}, \quad h_i >0, \quad i = 1, 2, \cdots, n,
\end{align}
and define the discrete inner product 
\begin{align}
    \l \mathbf{u}, \mathbf{v} \r_H := \mathbf{u}^T H \mathbf{v}=\sum_{i = 1}^n h_i \mathbf{u}_i \mathbf{v}_i.
\end{align}
We note that $h_i = h w_i >0$ with the non-dimensional constants $w_i>0$ being the weights of a composite quadrature rule and $h>0$ is the spatial step.
The following vectors are of repeated interest:
\begin{align}
     \bx^k \coloneqq (x_1^k, \dots, x_n^k)^T, &&
     \mathbf{0} \coloneqq (0, \dots, 0)^T, &&
     e_1 \coloneqq (1, 0,\dots , 0)^T, &&
      e_n \coloneqq (0, \dots , 0, 1 )^T.
\end{align}

We will also use the sequence spaces $\ell_2, \ellin$, those are the space of square summable sequences and doubly infinite bounded sequences.
That is 
\begin{align}
    \ell_2 &\coloneqq \{ s = s_1s_2 \cdots \ | \ s_i \in \R , \ \sum_i s_i^2 < \infty \}, \\
    \ellin &\coloneqq \{ s = \cdots s_{-1} s_0 s_1 \cdots \ | \ s_i \in \R , \ |s_i| < \infty \}.
\end{align}

\subsection{Measure Approximation}

Traditionally it has been the case that derivative operators have been approximated by FD quotients as an approximation of the limit given by a classical derivative. 
Finite element methods were then formulated to solve PDE's in their weak form, and provided a strong theoretical foundation for numerical methods. 
The SBP FD framework \cite{kreiss1974finite,strand1994summation} was then formulated as to mimic the theoretical guarantees of spectral methods for finite difference schemes. 
It was found that SBP methods also produced quadrature rules for numerical integration. 
Here, we generalise this idea by starting with an approximating discrete measure, and extrapolating differential operators in a flexible framework which allows for other critical components of the PDE to be approximated. 

We will always use an atomic measure to approximate Lebesgue measure. 
For simplicity we will restrict our attention to the interval $[-\pi, \pi]$ and denote the Borel measurable subsets to be $\mathcal{B}([-\pi, \pi]) $.

\begin{defn}
An atomic measure $\nu_n : \mathcal{B} ([-\pi, \pi]) \mapsto \R_{\geq 0}$ is a measure of the form 
\begin{align}
    \nu_n = \sum_{i=1}^n h_i \1_{(\cdot)} (x_i), 
\end{align}
where $h_i > 0 $ are quadrature weights and $x_i \in [-\pi, \pi]$ are the quadrature points. 
\end{defn}

A family of atomic measures $\{ \nu_n \}_{n=1}^{\infty}$ converges weakly to Lebesgue measure if for all continuous functions $f$ on $[-\pi , \pi]$ we have 
\begin{align}
    \int f d \nu_n \rightarrow \int_{- \pi}^{\pi} f . 
\end{align}
See that 
\begin{align}
    \int f d \nu_n = \sum_{i=1}^n h_i f (x_i). 
\end{align}
In particular, we may form the inner-product 
\begin{align}
    \l f,g \r_{\nu_n} \coloneqq \int fg d \nu_n = \l \mathbf{f}, \mathbf{g} \r_H.
\end{align}
So we may view the functions as discrete evaluations, or the measures as discrete approximations to the continuous setting. 
We make this clear as to inherit the theoretical framework from the finite element setting for finite difference operators. 
To gain this in our dual-pairing SBP framework we demand two conditions on our approximating measure $\nu_n$, these are:
\begin{enumerate}
    \item $\nu_n \xrightarrow{n \rightarrow \infty} dx$ in the weak sense; and,
    \item there exists linear operators $D_+, D_- : G_n  \mapsto G_n $ so that 
    \begin{align}
        \l D_+ \mathbf{f}, \mathbf{g} \r_H + \l  \mathbf{f}, D_- \mathbf{g} \r_H = B(fg). 
    \end{align}
\end{enumerate}
On the right hand side, $B: L^{\infty}([-\pi, \pi]) \mapsto \R$ is a boundary integral operator given through $B(f) \coloneqq f(\pi) - f(-\pi)$. 
This operator is identically zero on the interior of our domain granting the negative duality
\begin{align}
    D_+^T = - D_-,
\end{align}
when restricted to the interior of the domain. 
There are three primary advantages to using the dual-pairing SBP framework in the finite difference setting, these are:
\begin{enumerate}
    \item numerical stability is preserved from the continuous setting;
    \item weak derivatives are approximated instead of strong derivatives in traditional finite difference schemes; and,
    \item the stencils can be designed to mimic the dispersion relation in the continuous setting. 
\end{enumerate}
The first and second properties come from the general dual-pairing framework, the last is the primary topic of this paper for the design of such schemes. 
We briefly show the second property now. 
\begin{prop}
Let $f \in \mathcal{H}^2([-\pi, \pi])$ and $g \in C^\infty_0([-\pi,\pi])$, then as $n \to \infty$ we have that $D_+ f \mapsto f' \in L^2 ([-\pi,\pi])$.  
\end{prop}
\begin{proof}
From the assumptions above we have
\begin{align}
    \left| \l D_+f, g \r_{\nu_n} - \int f' g \right| = \left| \l f, D_- g \r_{\nu_n} - \int fg' \right|.
\end{align}
The right hand side converges to zero as $n$ tends to infinity. 
This is because
\begin{align}
    D_- g (x)\rightarrow g'(x),
\end{align}
strongly as $g$ is smooth and $D_-$ is of at least first order. 
Write $D_- g = g' + 
\varepsilon$ where minimally $\varepsilon = O(1/n) $, then the convergence is given from the weak convergence of $\nu_n$ as $f$ is necessarily absolutely continuous. 
\end{proof}
Thus even when the test function space comprises of smooth functions, the finite-difference operators $D_+, D_-$ may yield legitimate approximations to weak derivatives. 
In summary, the dual-pairing framework allows for weak-derivative approximation and the preservation of the summation-by-parts. 

\subsection{Discrete Dual-Pairing SBP Operators}

The matrix $B \coloneqq e_n e_n^T - e_1 e_1^T$ encodes the 1D boundary integral of the interval $[-\pi, \pi]$. 
Again, throughout our study, the matrix $H \in \R^{n \times n}$ is always a diagonal positive definite matrix. 
We now may state the integration-by-parts and the summation-by-parts properties. \\

For $u, v \in C^1([-\pi, \pi])$ we have the integration-by-parts property 
\begin{align}
    \l u' , v\r + \l u, v' \r = u(\pi)v(\pi) -u(-\pi)v(-\pi).
\end{align}
Let $D: \R^n \mapsto \R^n$ so that $\mathbf{u}' \approx D \mathbf{u}$ approximates the first derivative grid function $\mathbf{u}$.  The summation-by-parts property can be stated as 
\begin{align}
    \l  D \mathbf{u}, \mathbf{v} \r_H + \l \mathbf{u},   D \mathbf{v} \r_H = \mathbf{u}^T B \mathbf{v} = u_nv_n - u_1v_1,
\end{align}
for suitable grid functions $\mathbf{u}, \mathbf{v}$.
We may now separate the criteria of differentiation on a test function and the SBP criteria through considering the matrix equation
\begin{align}
    (H D)^T + H D = B.
\end{align}
This is the standard SBP framework \cite{kreiss1974finite, strand1994summation}. 
We will analyse the extension where we have two differential operators $D_+$, $D_-$ that obey on test function vectors 
\begin{align}
    \mathbf{u}' \approx D_+ \mathbf{u} && \mathbf{u}' \approx D_- \mathbf{u},
\end{align}
and the modified SBP property 
\begin{align}
    \l  D_+ \mathbf{u}, \mathbf{v} \r_H + \l \mathbf{u},   D_- \mathbf{v} \r_H = \l  \mathbf{u} , B \mathbf{v} \r,
\end{align}
that translates to the matrix equation \begin{align}
    (H D_+)^T + H D_- = B.
\end{align}
As is standard notation, define the matrices $Q_+ = H D_+$ and $Q_- = H D_-$.
To simplify finding these operators we will put the additional constraint 
\begin{align}
    D_+ = H^{-1}( \overline{Q}_+ + B/2) && D_- = H^{-1}( \overline{Q}_- + B/2). 
\end{align}
Note that
$$
H D_+= Q_+ = \overline{Q}_+ + B/2, \qquad H D_-= Q_- = \overline{Q}_- + B/2,
$$
and 
$$
\overline{Q}_+ = Q_+ - B/2, \qquad \overline{Q}_- = Q_- - B/2.
$$
We will assume that the stencil width for $D_+$ is $|r_2 + r_1|$ for $-r_1 \leq 0 < r_2$. 
When the matrix $H$ is diagonal, it is necessarily true that for $n >> q$ that $H$ is identical to the identity operator for $e_i$ with $i$ away from one or $n$.
For the points such that $He_i = e_i$ we call the position $i$ an interior point of the stencil. 
The compliment of this set is the boundary points of the operator. 
Throughout this study, we will use polynomials to be the test functions considered.
We may now formally define a dual-pairing SBP framework. 
\begin{defn}
A dual-pairing SBP operator of accuracy $q$ is the 3-tuple $(D_+, D_-, H)$ where $H > 0$ is a diagonal and positive matrix, and 
\begin{align}
    D_+ \mathbf{v} =  \mathbf{v}', \qquad D_- \mathbf{v} =  \mathbf{v}',
\end{align}
for polynomials of degree $q$ or less on the interior points and polynomials of degree $\lfloor q/2 \rfloor$ on the boundary and the involved matrices obey 
\begin{align}
    (H D_+)^T + H D_- = B.
\end{align}
\end{defn}

The high frequency modes (in the neighborhood of the $\pi$-modes) for both traditional stencils and the stencils we consider carries the majority of the numerical error. 
It is useful to consider some form of artificial dissipation for high-frequency components. 
For such scenarios we consider upwind dual-pairing SBP operators.

\begin{defn}
A upwind dual-pairing SBP operator is a dual-pairing SBP operator $(D_+, D_-, H)$ that additionally satisfies 
\begin{align}
    S \coloneqq \frac{\overline{Q}_+^T + \overline{Q}_+}{2},
\end{align} 
is negative semi-definite. 
\end{defn}

Before we can define a dispersion-relation-preserving-scheme we must recall some basic properties of the dispersion relation for linear hyperbolic problems. 

\subsection{$\alpha$-Dispersion Relation Preserving Schemes}
In this section we recall what the dispersion relation is for a simple hyperbolic system. 
We then define $\alpha$-dispersion-relation-preserving schemes ($\alpha$-DRP schemes).

Consider the simplest hyperbolic system
\begin{align}\label{eq:model_problem}
     \pfrac{v}{t}  = \pfrac{\sigma}{x} , \qquad \pfrac{\sigma}{t}  = \pfrac{v}{x} . 
\end{align}
Scaling constants and tensor products may be used to generalise our method to higher dimensional linear hyperbolic systems, but here we will only consider this simple case. 
Using the notation $\mathbf{u} = (v , \sigma)^T$, take the Fourier solution basis of the form 
\begin{align}
    \mathbf{u}_{\omega,k}(x,t) = \mathbf{u}_0 \exp(-\mi (\omega t -  kx)). 
\end{align}
Here $\omega \in \mathbb{R}$ is the angular frequency, $k \in \mathbb{R}$ is the spatial wave number  and $\mathbf{u}_0 \in \mathbb{R}^2$ is the constant polarisation vector.
For the continuous model problem \eqref{eq:model_problem}, the wave-mode $\mathbf{u}_{\omega,k}$ belongs to the solution if and only if 
\begin{align}\label{eq:dispersion_relation}
    \omega^2 =  k^2, \quad k \in \R.
\end{align}
This is the dispersion relation for this model problem \eqref{eq:model_problem}.
\begin{remark}
We can take square roots in \eqref{eq:model_problem} yielding the linear dispersion relation  
\begin{align}\omega = k, \quad k \in \mathbb{R}.
\end{align}
\end{remark}
For the current study, we are only interested in analysing the dispersion relation for the interior of a given stencil. 

We consider the $2\pi$-periodic functions in interval $x \in [-\pi, \pi]$ and  the uniform discretisation \eqref{eq:grid_points} of the interval.

Note that due to periodicity of the grid functions $\mathbf{v}$ we must have
\begin{align}
\mathbf{v}_{1-j} = \mathbf{v}_{n-j}, \quad \mathbf{v}_{n+j} = \mathbf{v}_{1+j}, \quad j = 0, 1, 2, \cdots, n-1.
\end{align}

\begin{defn}
Let $(D_+, D_- ,H)$ be a dual-pairing SBP operator, and let $D_{\eta,\circ}$ be the interior stencil for $\eta \in \{-,+\}$.
Let $\ellin$ be the space of doubly infinite bounded sequences of real numbers, then $D_{\eta,\circ} :\ellin \mapsto \ellin$.
These operators are the \emph{interior upwind operators}.
\end{defn}

Recall for interior points $He_i = e_i$, where $e_i$ is the $i$-th unit vector,  and so for the interior operators we only need 
\begin{align}
    D_{+,\circ}^T + D_{-,\circ} = \mathbf{0},
\end{align}
for the summation-by-parts criteria. 

Now let us consider the semi-discrete counterpart
\begin{align}\label{eq:1D_wave_disc}
      \frac{d \mathbf{v}}{dt} = D_{+,\circ} \boldsymbol{\sigma}, \qquad \frac{d \boldsymbol{\sigma}}{dt} =D_{-,\circ} \mathbf{v},
\end{align}
where we have replaced the spatial derivatives with the finite difference operators $D_{+,\circ}, D_{-,\circ}$.

We consider the interior stencils only, thus
$$
\left(D_{+,\circ} \mathbf{v}\right)_j = \frac{1}{h}\sum_{l=-r_1}^{r_2} \alpha_l v_{j+l}, \qquad
\left(D_{-,\circ} \mathbf{v}\right)_j = \frac{1}{h}\sum_{l=-r_2}^{r_1} \beta_l v_{j+l}, \qquad
    0\le r_1 < r_2,
$$
with the consistency requirements
$$
\sum_{l=-r_1}^{r_2} \alpha_l = \sum_{l=-r_1}^{r_2} \beta_l = 0, \qquad \sum_{l=-r_1}^{r_2} l\alpha_l = \sum_{l=-r_2}^{r_1} l\beta_l = 1.
$$
Here, $h>0$ is the uniform grid spacing, $\alpha_l, \beta_l$ are the non-dimensional constant coefficients defining the upwind finite difference stencils. Note that $\beta_l = -\alpha_{-l}$ for $l = -r_2, -(r_2-1), \cdots r_1$.

Let $\bom_N : [-\pi ,\pi] \mapsto \R$ be the numerical dispersion relation for a given discrete operator for the semi-discrete approximation \eqref{eq:1D_wave_disc}.
For a given dual-paring SBP scheme we may find an explicit expression for the discrete dispersion relation. 

Let $(D_{+, \circ}, D_{-, \circ} )$ be an internal upwind pairing, then introduce $\widetilde{k} = hk$ the numerical wave number and define the discrete Fourier symbols
\begin{align}\label{eq:Disc_Fourier_Symbol}
    \bom_{+} \coloneqq \sum_{l = -r_1}^{r_2} \alpha_l (\cos(l\widetilde{k}) + \mi \sin(l\widetilde{k})), && \bom_{-} \coloneqq-\sum_{l = -r_1}^{r_2} \alpha_l (\cos(l\widetilde{k}) - \mi \sin(l\widetilde{k})),
\end{align}
such that
\begin{align}\label{eq:1D_wave_disc_Fourier_0}
      h\left( D_{+, \circ} \mathbf{v}\right)_j = \bom_{+} \mathbf{v}_j(t), \qquad h\left(D_{-, \circ} \mathbf{v}\right)_j = \bom_{-} \mathbf{v}_j(t),
\end{align}
where 
$$
 \mathbf{v}_j(t)= \widehat{v}_0 \exp(-\mi (\omega t -  k x_j)).
$$
\begin{lemma}
Consider the semi-discrete approximation \eqref{eq:1D_wave_disc} with $(D_{+, \circ}, D_{-, \circ} )$. The semi-discrete discrete dispersion relation for \eqref{eq:1D_wave_disc} is given by
\begin{align}
    \bom^2_N = -\bom_{+}\bom_{-}=\left(\sum_{l=-r_1}^{r_2} \alpha_l \cos(l \widetilde{k})\right)^2 +  \left(\sum_{l=-r_1}^{r_2} \alpha_l \sin(l \widetilde{k})\right)^2,
\end{align}
where $\bom_N = h\omega$ is numerical frequency and $\widetilde{k} = hk$ is the numerical wave number.
\end{lemma}

\begin{proof}
Inserting the plane waves
$$
 \mathbf{v}_j(t)= \widehat{v}_0 \exp(-\mi (\omega t -  k x_j)), \quad  \boldsymbol{\sigma}_j(t)= \widehat{\sigma}_0 \exp(-\mi (\omega t -  k x_j))
$$
in the semi-discretised PDE \eqref{eq:1D_wave_disc}, we have
\begin{align}\label{eq:1D_wave_disc_Fourier}
      -\mi\omega_N\mathbf{v}_j = \bom_{+} \boldsymbol{\sigma}_j, \qquad -\mi\omega_N \boldsymbol{\sigma}_j  = \bom_{-}\mathbf{v}_j,
\end{align}
and the solvability condition
\begin{align}
    \text{det} \left(  
    \begin{pmatrix}
    - \mi \bom_N  && -  \bom_{+}\\
    -  \bom_{-} && - \mi \bom_N 
    \end{pmatrix}
    \right) = 0 \iff \bom^2_N + \bom_{+} \bom_{-} = 0.
\end{align}
Re-arrangement of the terms to find $\bom_N$ and using $\bom_{+}$, $\bom_{-}$ as defined in \eqref{eq:Disc_Fourier_Symbol} gives the result. 
\end{proof}

We define the numerical dispersion relation
\begin{align}\label{eq:numerical_dispersion_relation}
    \bom^2_N(k) =\left(\sum_{l=-r_1}^{r_2} \alpha_l \cos(l{k})\right)^2 +  \left(\sum_{l=-r_1}^{r_2} \alpha_l \sin(l{k})\right)^2, \qquad {k} \in [-\pi, \pi].
\end{align}

%
Similarly, for the traditional SBP operator we have the numerical dispersion relation
 \begin{align}\label{eq:dispersion_relation_traditional}
      \bom^2_N(k) =  \left(\sum_{j=1}^{r} 2\gamma_j \sin(j {k})\right)^2, \qquad {k} \in [-\pi, \pi],
\end{align}
where $\gamma_j$ are the non-dimensional constant coefficients of the finite difference operator, with $\gamma_{-j} = -\gamma_j$, $\gamma_0 =0$, and satisfying the consistency requirements
$$
\sum_{l=-r}^{r} \gamma_l  = 0, \qquad \sum_{l=1}^{r} 2l\gamma_l = 1.
$$
By comparing the numerical dispersion relation $\bom_N(k)$, defined in \eqref{eq:numerical_dispersion_relation}--\eqref{eq:dispersion_relation_traditional} and the continuous dispersion relation $\omega(k)$, defined in \eqref{eq:dispersion_relation}, we can determine the numerical dispersion error. 
Define the point-wise dispersion error
\begin{align}
    \varepsilon_d(\bom_N, k)  = | \omega(k)  - \bom_N (k) |, \quad  {k} \in [-\pi, \pi].
\end{align}
We are interested in schemes that are accurate across the whole spectrum ${k} \in [-\pi, \pi]$, so we define the maximal relative dispersion error to be 
\begin{align}
    \varepsilon_{\infty}(\bom_N) = \max_{k \in [-\pi, \pi]} \frac{\varepsilon_d (\bom_N , k)}{\omega(k)}.
\end{align}
Note that for the traditional SBP operators based on central finite difference scheme, from \eqref{eq:dispersion_relation_traditional} we have $\bom_N(\pm\pi) = 0$.
Therefore, all central stencil schemes have a maximal dispersion error of $100 \%$, and are completely unable to resolve the highest frequency modes on a given computational grid \cite{tam1993dispersion,linders2015uniformly}.

We are interested in constructing schemes whose dispersion error for high frequencies is acceptable for most application problems. This motivates our definition for a $\alpha$-DRP scheme. 

\begin{defn}
Let $(D_{+, \circ}, D_{-, \circ} )$ be an interior dual-pairing SBP operator. 
We call $(D_{+, \circ}, D_{-, \circ} )$ an $\alpha$-DRP operator if the maximal relative dispersion error is bounded by $\alpha$.
When a dual-pairing SBP operator has a $\alpha$-DRP operator on its interior, we call it a $\alpha$-DRP dual pairing SBP operator. 
\end{defn}
For many applications, the tolerance $\alpha = 5 \% $ is an acceptable error. 
Colloquially, we will call a scheme DRP-preserving if it is a $\alpha$-DRP scheme with $\alpha < 5 \%$.
In the next section we show how we can construct such a scheme.

\section{Construction of $\alpha$-Dispersion Relation Preserving Schemes}

In this section we show how to construct DRP-schemes which are provably stable through the dual-pairing SBP framework. 
Previously work \cite{tam1993dispersion,linders2015uniformly} has looked at optimising the dispersion relation for central stencil schemes. 
All such attempts are $100 \%$-DRP schemes that support spurious wave modes in their solution, being unable to resolve wave numbers higher than $\pi/2$. 
Due to the dual-pairing framework, we may resolve the full spectrum to $5\%$ error. 

\subsection{Interior Stencil}
First we focus on designing the interior stencil of our scheme. 
Recall we are interested in the minimisation 
\begin{align}
     \argmin_{ \bom_N \in \Omega }  | \omega(k)  - \bom_N (k) |
\end{align}
for a discrete dispersion relation $\bom_N$ from a family of candidate models $ \Omega$ and all $k \in [- \pi, \pi]$. 
Ideally, we would consider the infinity norm of this quantity but for computational feasibility we will consider the least-squares minimisation on the whole spectrum.
To the authors knowledge, this is the first successful attempt to optimise the dispersion error for the whole spectrum. 
Note that continuous angular frequency $\omega(k)$ and the numerical counterpart $\bom_N (k)$ take the sign of its argument $k\in[-\pi, \pi]$.
We  consider the following least squares optimisation
\begin{align}
    \arg \min_{\bom_N \in \Omega} \| \omega - \bom_N \|_2^2,
\end{align}
for an appropriate family of candidate models $\bom_N$. 
It is important that our approximant is exact around $k = 0$, as this determines the error convergence with grid refinement. 

As stated above, for convenience we will work with quadratic positive functions $\omega^2(k)$, $\bom_N^2(k)$, however, using the following Lemma the analysis holds also for $\omega(k)$, $\bom_N(k)$.
%
%
\begin{lemma}\label{Lem:Lemma1}
Let $\omega_1,\omega_2 \in L^2 ([-\pi, \pi])$ and $\omega_1(k)$, $\omega_2(k)$ that take the sign of its argument $k\in [-\pi, \pi]$, then we have
\begin{align}
    \| \omega_1 - \omega_2 \|_2^2 \leq \sqrt{2 \pi} \| \omega_1^2 - \omega_2^2 \|_2.
\end{align}
\end{lemma}
\begin{proof}
First by the reverse triangle inequality and non-negativity,
\begin{align}
    \int | \omega_1 -\omega_2 | \leq \int |\omega_1 + \omega_2|.
\end{align}
Now we have 
\begin{align}
    \int |\omega_1-\omega_2|^2 \leq \int |\omega_1 - \omega_2| |\omega_1+\omega_2| = \int |\omega_1^2 - \omega_2^2|.
\end{align}
So we gain
\begin{align}
    \| \omega_1 - \omega_2 \|_2^2 \leq \| \omega_1^2 - \omega_2^2 \|_1,
\end{align}
but as we are compactly supported:
\begin{align}
    \| \omega_1^2 - \omega_2^2 \|_1 \leq \sqrt{2 \pi} \| \omega_1^2 - \omega_2^2 \|_2.
\end{align}
\end{proof}
If we now consider $\Omega$ to be a convex set, then both 
\begin{align}
    \arg \min_{\bom_N \in \Omega } \| \omega - \bom_N \|_2,
\end{align}
and 
\begin{align}
    \arg \min_{\bom_N \in \Omega} \| \omega^2 - \bom^2_N \|_2,
\end{align}
are strictly convex and share the same minimiser if it is in the candidate model set due to Lemma \ref{Lem:Lemma1}. 

Now the function to be approximated, $k \mapsto k^2$, is smooth around zero.
To make our optimisation computationally feasible, we will place both $\omega$ and our elements of $\bom_N \in \Omega$ into orthonormal cosine bases. That is
\begin{align}
    \bom^2_N = \sum_{i} (\boldsymbol\beta_N)_i \cos(k i ), \quad \omega^2 = \sum_{i} (\beta)_i \cos(k i ), \quad k \in [-\pi, \pi],
\end{align}
where $ (\boldsymbol\beta_N)_i, (\beta)_i \in \mathbb{R}$ are the Fourier coefficients.
This is a natural choice for the function we are aiming to approximate.
The coefficients $(\beta)_i$ for the continuous dispersion relation have a closed form expression, namely
$$
(\beta)_i = \frac{4(-1)^{i}}{i^2}.
$$
The coefficients $(\boldsymbol\beta_N)_i$ for the numerical dispersion relation depend on the finite difference stencils and can be computed using a symbolic mathematical software, such as Maple. The coefficients form a finite square-summable sequence, and vanishes after a finite $i\ge j$, that is
$$(\boldsymbol\beta_N)_i =0, \qquad \forall  i \ge j.$$

\begin{lemma}\label{lem:trun}
For positive functions $\omega, \bom \in L^2([-\pi, \pi])$,
\begin{align}
    \| \bom - \omega \|_2^4 \leq 2 \pi \| \boldsymbol\beta - \beta_j \|_2^2 + \mathbb{T}(j)
\end{align}
where $\boldsymbol\beta \in \ell^2_{0,j}$ the square summable space of sequences whose entries are zero after $j$ entries, $ \beta \in \ell^2$ are such that
\begin{align}
    \bom^2 = \sum_{i} (\boldsymbol\beta)_i \cos(k i ), \quad \omega^2 = \sum_{i} (\beta)_i \cos(k i ), \quad k \in [-\pi, \pi],
\end{align}
 with $\beta_j$ the projection of $\beta$ into $\ell^2_{0,j}$ and
\begin{align}
    \mathbb{T}(j) = O(1/j^4).  
\end{align}
\end{lemma}
\begin{proof}
Applying Lemma \ref{lem:trun} and Parseval's formula yields 
\begin{align}
    \| \bom - \omega \|_2^4 \leq 2 \pi \| \boldsymbol\beta - \beta \|_2^2.
\end{align}
Now by Bessel's equality for $(\boldsymbol\beta - \beta)_{\leq j}$ the $j^{th}$ level truncation of the sequence and $(\boldsymbol\beta - \beta_j)_{>j}$ its compliment we have
\begin{align}
    \| \boldsymbol\beta - \beta \|_2^2 = \| (\boldsymbol\beta - \beta)_{\leq j} \|_2^2 + \| (\boldsymbol\beta - \beta_j)_{>j} \|_2^2.
\end{align}

For $j' > j$ the entry in $(\boldsymbol\beta - \beta)_{j'}$ decays like $1/{(j')^2}$ as the extended periodic function $\omega^2 = \sum_{l \in \Z} (k- 2 \pi l)^2 \1_{[l - \pi, l + \pi) } $ on $\R$ is once weakly differentiable, and so the square sum of these terms behaves like $1/j^4$ completing the claim.
\end{proof}

Last, it would be beneficial if we could find a parameterised family that forms a linearly independent basis set. 
For instance if $\bom_N^2: [-\pi , \pi] \mapsto \R$ is of the form 
\begin{align}
    \bom_N^2(k) = \sum_{j=0}^N \beta_j \cos(j k),
\end{align}
for $N \in \{ a, a+1, \dots, b \}$, then the functions $\bom_N$ form a linearly independent basis for the family 
\begin{align}
    \Omega = \left\{ \sum_{r} \gamma_r \bom_r^2 \ | \ \gamma_r \in \R \right\}.
\end{align}
Further if $a = 0$ and as $b \rightarrow \infty$, $\overline{\lim_{b \rightarrow \infty } \Omega} = L^2 ([-\pi, \pi])$.
This gives a approximation space for $L^2([-\pi, \pi])$, but it is still of great importance that the approximant $\bom^2$ is uniformly a good approximant to $\omega$ around zero. 
For our stencils, we fix $a>0$ so for each member of the generated $\bom$ uniformly approximates $\omega$ around $k=0$, then use the remaining free-parameters to approximate the function in the higher-frequency region. \\

All that is left is to find a family $\Omega$ which obeys our given assumptions.
Let $D_p$ be the central difference stencil of order $p$, then following \cite{Mattsson2017} we may construct 
\begin{align}
    (D_{+})_p \coloneqq D_p - \frac{h^{p+1} }{ \kappa_p } (\Delta_+ \Delta_-)^{p/2 + 1} - \frac{h^{p} }{ \kappa_{p-1} } \Delta_-(\Delta_+ \Delta_-)^{p/2 - 1},
\end{align}
where $\kappa_p$ is the relevant scaling constant with $\Delta_+, \ \Delta_-$ being the standard first order backward and forward FD operators, see \cite{gustafsson1995time}, defined by 
$$
\Delta_+\mathbf{v}_j = \frac{\mathbf{v}_{j+1}-\mathbf{v}_j}{h}, \quad  \Delta_-\mathbf{v}_j = \frac{\mathbf{v}_{j}-\mathbf{v}_{j-1}}{h}.
$$
This produces a finite-difference approximant of order $p$, taking the negative transpose of $ D_{-}^p$ gives an order $p$ approximant that obeys 
\begin{align}
     ({D_{+, \circ}})_p^T +  (D_{-, \circ})_p = \mathbf{0}.
\end{align}
The co-efficients for these operators are given in Table \ref{tab:upwind}. 

Define 
\begin{align}
    \Omega_{+}^{a,b} &\coloneqq \{ \bom_+ \ | \ \bom_+ \text{ is of order $p \in [a,b]$ from Table \ref{tab:upwind}} \}, \\
    \Omega_{-}^{a,b} &\coloneqq \{ \bom_- \ | \ \bom_- \text{ is of order $p \in [a,b]$ from Table \ref{tab:upwind}} \}, \\
    \Omega^{a,b} &\coloneqq \{ \bom \ | \  \bom^2 = ( \sum_i \gamma_i \bom_{+}^{(i)}) ( \sum_i \gamma_i \bom_{-}^{(i)}), \ \bom_{+}^{(i)} \in \Omega_{+}^{a,b}, \ \bom_{-}^{(i)} \in \Omega_{-}^{a,b}, \ \sum_i \gamma_i = 1 \}.
\end{align}
This family is naturally parameterised by the parameters $\{ \gamma_i \}$.
The following is a result found from symbolic computational verification. 
\begin{lemma}\label{lem:coscheck}
Let $\bom \in \Omega^{a,b}$ for any $a \leq b$ for $a,b \in \{2, \dots, 9 \}$, then $\bom$ is represented exactly in the form 
\begin{align}
    \bom^2 = \sum_{j = 0}^N  c_j \cos(j k),
\end{align}
for some finite $N$. 
\end{lemma}
\begin{proof}
It is sufficient to check 
\begin{align}
    \bom_+^{(i)}  \bom_-^{(j)}  +  \bom_+^{(j)}  \bom_-^{(i)},
\end{align}
is of the correct form for all $i,j$ as $\bom^2$ is made from weighted linear combinations of these. 
For $\Omega^{a,b}$ there are $\binom{b-a}{2} = \frac{(b-a)(b-a-1)}{2}$ checks to complete, which is easily computed in a symbolic manipulation software. 
\end{proof}

Putting all of our work together we can prove the following bound.
\begin{theorem}
Let $\bom, \omega \in L^2([-\pi, \pi])$ be positive functions whose squares are once weakly differentiable and have Fourier co-efficients $\boldsymbol{\beta}, \beta \in \ell_2$, then
\begin{align}
\|\omega - \bom \|_{2}^4 \leq \sqrt{2 \pi } \| \beta_l - \boldsymbol{\beta}_l \|_2^2 + O(1/l^4),
\end{align}
for $\beta_l, \boldsymbol{\beta}_l$ the $l^{th}$ level truncation of $\boldsymbol{\beta}, \beta $.
\end{theorem}

Currently, we have the strong assumption that our actual constraint set $\Omega^{a,b}$ has the form $L^{\infty}([-\pi, \pi]) \cap \mathcal{C}$, for a convex set $\mathcal{C}$. 
This is not in general true, but we can note that 
\begin{align}
    (\boldsymbol\beta - \beta_j)_l = \sum_{i,j} c_{i,j} \gamma_i \gamma_j + c_{0,0},
\end{align}
where we have the constraint set $\sum_i \gamma_i = 1$. 
An appropriate convex relaxation of this problem is to consider the cost of the form
\begin{align}
    (\boldsymbol\beta - \beta_j)_l = \sum_{i,j} c_{i,j} \pi_{i,j} + c_{0,0},
\end{align}
for $\pi_{i,j}$ in $\R$ such that $\pi_{i,j} \leq \gamma_i \gamma_j$ and $\sum_i \gamma_i = 1$. 
Such a relaxation still has uniqueness and existence of minimiser, but as soon as we introduce the constraint of equality on  $\pi_{i,j} = \gamma_i \gamma_j$ we lose uniqueness. 

\begin{ex}
Consider the case $N = 2$ with the dummy variables $\pi_{1,1}, \pi_{1,2}, \pi_{2,2}$ which obey the relations 
\begin{align}
    \pi_{1,1} = \gamma_1^2,&& \pi_{1,2} = \gamma_1\gamma_2, && \pi_{2,2} = \gamma_2^2.
\end{align}
We will consider the three dimensional space spanned by $\pi_1, \pi_2, \pi_3$. 
Recall $\gamma_1, \gamma_2 \in \R$ granting that the constraints given correspond to a parametric parabloid, meanwhile the constraint 
\begin{align}
    \gamma_1 + \gamma_2 = 1,
\end{align}
is equivalently the plane given through the implicit equation
\begin{align}
    (\pi_{1,1}, \pi_{1,2}, \pi_{2,2}) \mapsto (\pi_{1,1}, \pi_{1,2}, 1 - \pi_{1,1} -  2 \pi_{1,2}).
\end{align}
The intersection of these curves corresponds to a one-dimensional subspace $P$ given through the parametric equation 
\begin{align}
    (\pi_{1,1}, \pi_{1,2}, \pi_{2,2}) = (u^2, u(1 - u), (1 - u)^2),
\end{align}
for $u \in \R$. 
Now consider the cost function 
\begin{align}
    C(\pi_{1,1}, \pi_{1,2}, \pi_{2,2}) = (\pi_{1,1}-1)^2+ \pi_{1,2}^2+ (\pi_{2,2}-1)^2.
\end{align}
This is a convex function in the parameters $\pi_{1,1}, \pi_{1,2}, \pi_{2,2}$, but our constrained optimisation yields two global minima. 
Consider the sublevel sets 
\begin{align}
    C_{\varepsilon} \coloneqq \{ (\pi_{1,1}, \pi_{1,2}, \pi_{2,2}) \ | \ C(\pi_{1,1}, \pi_{1,2}, \pi_{2,2}) \leq C(\pi_{1,1}, \pi_{1,2}, \pi_{2,2}) \},
\end{align}
where it is clear $C_{\varepsilon}$ is closed and convex for all $\varepsilon \geq 0 $ as $C$ is convex. 
We may reform our optimisation criteria as 
\begin{align}
    \min_{\varepsilon \geq 0 } \{ \varepsilon \ | \ C_{\varepsilon} \cap P \neq \emptyset \}.
\end{align}
This occurs at $\varepsilon = 1$ showing the minimum is one, and there are two points of intersection corresponding to $(\gamma_1, \gamma_2)$ being either $(1,0)$ or $(0,1)$, granting two minima. 
This occurs as the constraint set $P$ is not convex. 
For the $N=2$ case, by B\'ezout's bound we are guaranteed at most two solutions as there are at most four points of intersection of these quadratic curves, but when counted with multiplicity there are at most two intersection points due to the minimality condition. 
\end{ex}

Finally, we can state the (finite-dimensional) optimisation to be performed.

\begin{align}
    \argmin_{ \bom \in \Omega^{a,b} } \| \boldsymbol\beta(\boldsymbol\pi) - \beta_j \|_2^2 
\end{align}
subject to 
\begin{align}
    \pi_{i,j} = \gamma_i \gamma_j, && \sum_{i} \gamma_i = 1.
\end{align}

The cost function is now convex and coercive in the variables $\pi_{i,j}$, granting existence of a minimiser as it positive, but due to the non-convex constraints, we may lose uniqueness.
Using standard least-squares optimisation techniques, we may find global minima. 
The results for operators of orders $4,5,6,7$ are presented in the Appendix.
These are presented in rational arithmetic.\\

By a simple bootstrap argument, we may also consider the minimisation 
\begin{align}
    \argmin_{ \bom \in \Omega^{a,b} } \| \boldsymbol\beta(\pi) - \beta_j \|_\rho^2 
\end{align}
subject to 
\begin{align}
    \pi_{i,j} = \gamma_i \gamma_j, && \sum_{i} \gamma_i = 1,
\end{align}
for a weighted norm $\| \cdot \|_{\rho}$ that is induced by the inner product \begin{align}
    \l u, v \r_{\rho} = \int uv \rho,
\end{align}
for $\rho >0$. 
This can be done through noting that for $\rho \in L^2([-\pi, \pi])$ we have 
\begin{align}
    \rho(k) = \sum_{j =0 }^\infty \rho_j \cos(jk),
\end{align}
which can be formulated with an identical procedure.
Of particular interest is the weight function 
\begin{align}
    \rho(k) = 
    \begin{cases}
    1 & \text{ if } |k| \leq k'\\
    0 & \text{ otherwise,}
    \end{cases}
\end{align}
for some maximum desired frequency $k' \in [\pi/2,\pi]$. 
Such an optimisation would yield operators which are more accurate for lower frequencies at the expense of high-frequency accuracy.
We can easily target which part of the spectrum we wish to have the most accuracy with this bootstrap. 
This is not the focus for the current study, and so our $\alpha$-DRP operators will be made with the weight $\rho = 1$, but we will give a small example. 

\begin{ex}\label{ex:rho}
Consider the weight function $\rho(k) = \exp(0.3\cdot k^2)$ which weights the importance of the high-frequencies higher than that of the low-frequencies. We will use the approximation space $\Omega^{2,9}$ so in this case the $a$ value is small and the accuracy of this operator for low frequencies would not be suitable for application. 
The objective for our optimisation has the form
\begin{align}
    \min_{ \bom \in \Omega^{2,9} } \| \boldsymbol\beta(\pi) - \beta_j \|_\rho^2 = \min_{ \bom \in \Omega^{2,9} } \sum_{j} (\sum_{i\leq j} c_{\rho,i,j} \pi_{i,j} )^2,
\end{align}
where $c_{\rho,i,j}$ are appropriately weighted by the the Fourier cosine representation of $\rho$.
Let $\bom^*$ be the found dispersion relation, then we can compute the point-wise relative error to be $\varepsilon(k) = |\bom^*(k) - \omega(k)|/\omega(k)$. 
We plot and compare this relative error with an identical optimisation where the weight $\rho \equiv 1$ was used. 
\begin{figure}[H]
\centering
  \noindent
  \makebox[\textwidth]{\includegraphics[width =1.2\textwidth]{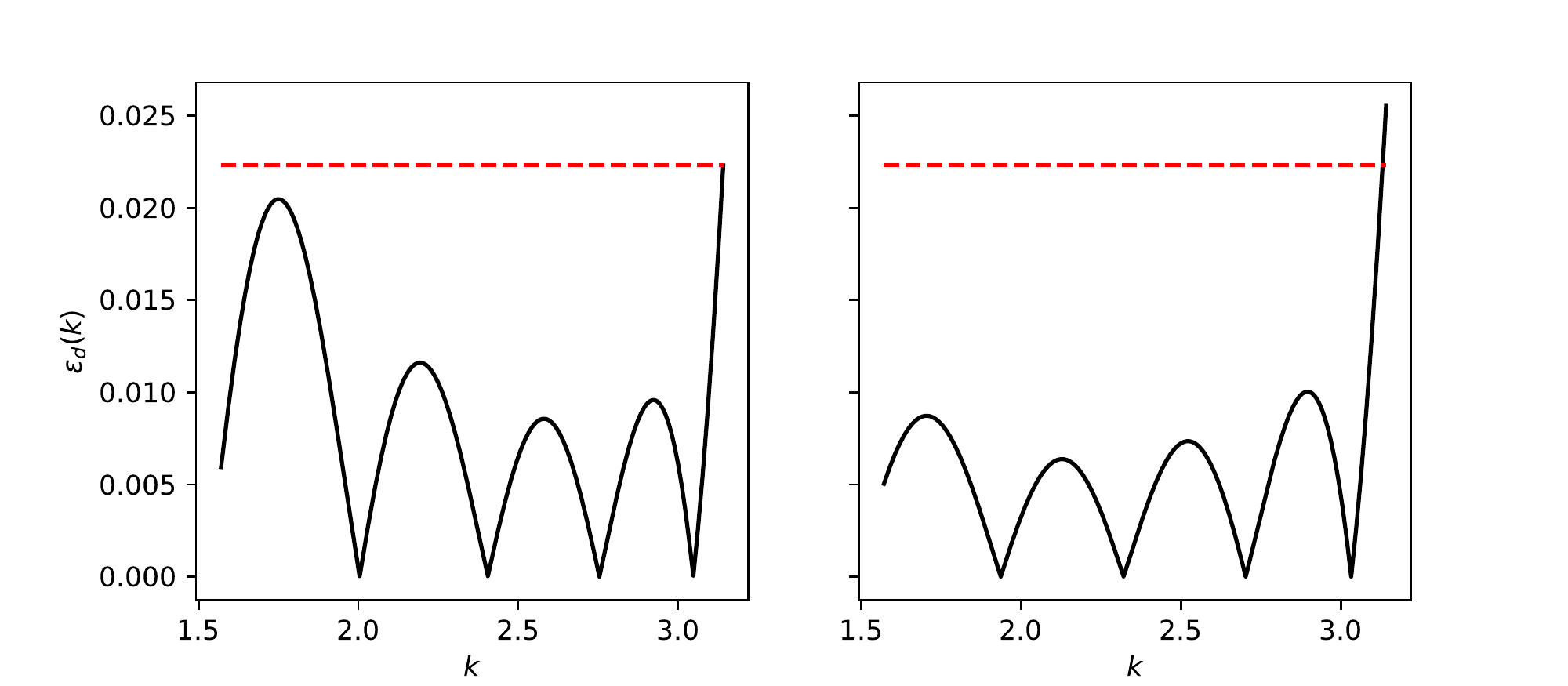}}%
  \caption{Dispersion errors on $[\pi/2, \pi]$ with the maximal dispersion error of the high-frequency scheme plotted in red.}
\end{figure}
As we can see, our weight function penalises the higher-end of the spectrum more and yields a smaller error for the $\pi$-mode at the expense of worse error bounds for lower frequencies. 
\end{ex}

\begin{remark}
A heuristic observed for designing interior operators is that higher accuracy can be found with increasing $a$ and $b$ for different parts of the spectrum.
We observe:
\begin{enumerate}
    \item higher $a$ values correspond to higher-order test functions and generally lead to operators with better accuracy in the low-frequency region $[0, \pi/2]$;
    \item the $b$ value can be chosen sufficiently larger than $a$ to allow the free variables in the linearly independent cosine approximating functions to minimise the error in the high-frequency $[\pi/2, \pi]$ region.
\end{enumerate}
In Example \ref{ex:rho}, there is a larger error in the low-frequency region than the high-frequency region as $a$ is small but $b >> a$. 
\end{remark}

We may also use the weight $\rho$ to be more accurate at the high-frequency end of the spectrum, at the expense of accuracy for low frequencies. 
This sacrifice could be compensated for by setting the value $a$ sufficiently large. \\

Last, we consider the simplest dispersion relation here, but this can be generalised to more complicated relations in the same framework. 
For instance, in fluid dynamics dispersion relations of the form $\omega(k)^2 = k \tanh{k}$ are common.
In Figure \ref{fig:drptanh} is an approximant we have made for this case.
This is a second order operator, common in fluid modelling, with three additional stencil points. 
The maximal dispersion error is no greater than $2\%$. 

\begin{figure}
    \centering
    \includegraphics[width = \textwidth]{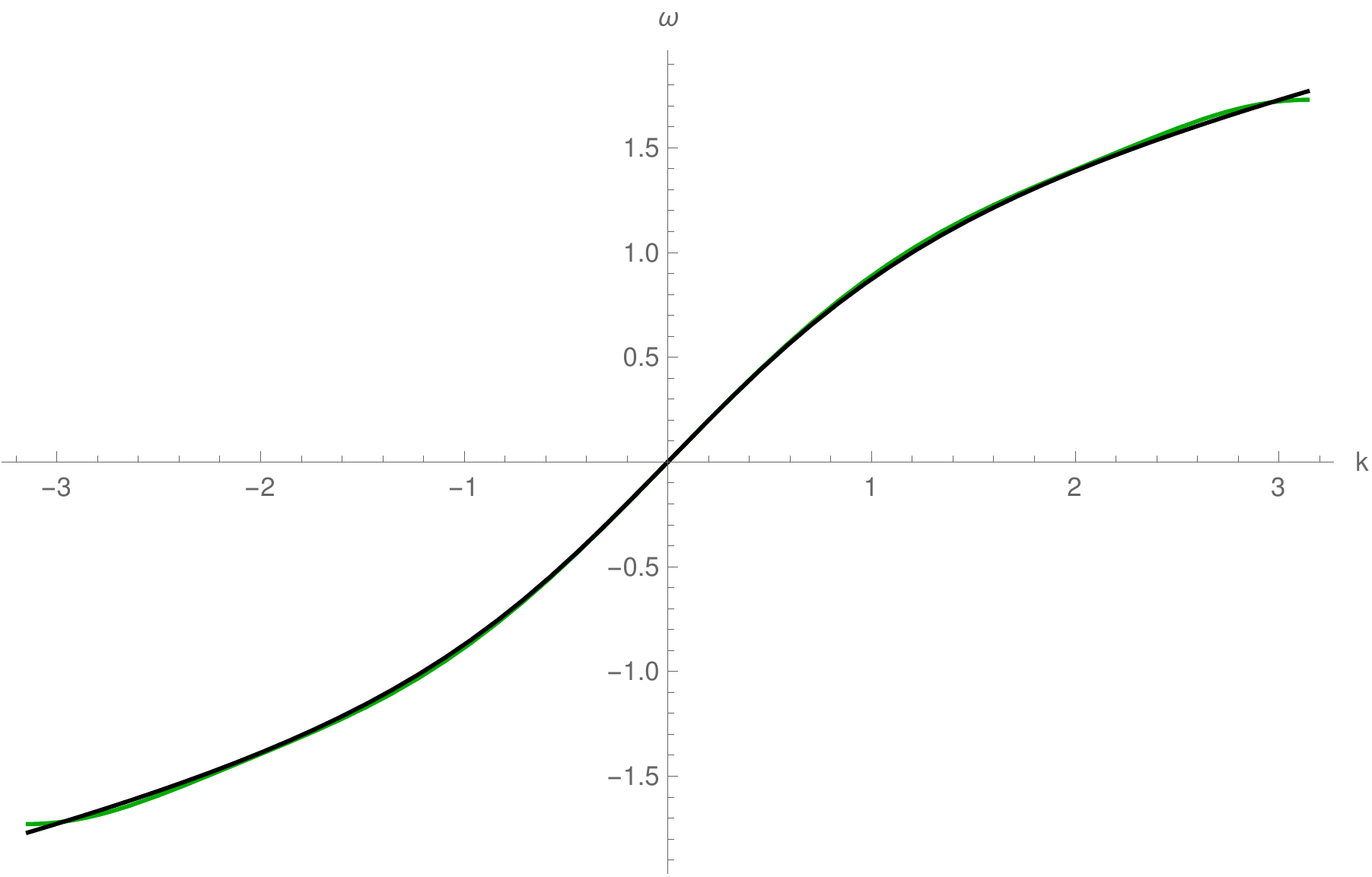}
    \caption{A second order dispersion relation preserving scheme for $k \to \omega(k)^2 = k \tanh{k}$.
    This function is difficult to approximate due to the poor Fourier representations of $\tanh{k}$ but can be done through Stejiles representations.}
    \label{fig:drptanh}
\end{figure}

\subsection{Boundary Stencil}

In this section we show how to close the boundary stencil for a given interior stencil so that we keep the dual-pairing SBP framework. 
As we have used upwind stencils for our interior scheme, we will close the boundaries so that the final dual-pairing SBP scheme is also an upwind scheme. 
We assume the ansatz 

\begin{align}
    (Q_+)_{i,j} = 
\begin{cases}
q_{i,j} ,& \text{ if } i,j \leq s\\
q_{s+1-j,s+1-i} ,& \text{ if } i,j \geq n - s\\
q_k ,& \text{ else if } i = j - k \\
0, & \text{ otherwise,}
\end{cases}
\end{align} 
where $k \in \{r_1, \dots, r_2 \}$. 
For example, in the case $ n = 9, s = 4, r_1 = -2, r_2 = 3$ we have 
\begin{align}
    Q_+ = 
    \left( \begin {array}{ccccccccc} q_{{1,1}}&q_{{1,2}}&q_{{1,3}}&q_{{1,
4}}&0&0&0&0&0\\ \noalign{\medskip}q_{{2,1}}&q_{{2,2}}&q_{{2,3}}&q_{{2,
4}}&q_{{3}}&0&0&0&0\\ \noalign{\medskip}q_{{3,1}}&q_{{3,2}}&q_{{3,3}}&
q_{{3,4}}&q_{{2}}&q_{{3}}&0&0&0\\ \noalign{\medskip}q_{{4,1}}&q_{{4,2}
}&q_{{4,3}}&q_{{4,4}}&q_{{1}}&q_{{2}}&q_{{3}}&0&0\\ \noalign{\medskip}0
&0&q_{{-2}}&q_{{-1}}&q_{{0}}&q_{{1}}&q_{{2}}&q_{{3}}&0
\\ \noalign{\medskip}0&0&0&q_{{-2}}&q_{{-1}}&q_{{4,4}}&q_{{3,4}}&q_{{2
,4}}&q_{{1,4}}\\ \noalign{\medskip}0&0&0&0&q_{{-2}}&q_{{4,3}}&q_{{3,3}
}&q_{{2,3}}&q_{{1,3}}\\ \noalign{\medskip}0&0&0&0&0&q_{{4,2}}&q_{{3,2}
}&q_{{2,2}}&q_{{1,2}}\\ \noalign{\medskip}0&0&0&0&0&q_{{4,1}}&q_{{3,1}
}&q_{{2,1}}&q_{{1,1}}\end {array} \right)
\end{align}
Similarly, we make the ansatz for the weight matrix $H>0$ to be
\begin{align}
    H_{i,j} = 
    \begin{cases}
    h_{i} ,& \text{ if } i=j,\  i \leq s\\
    h_{s+1-j} ,& \text{ if } i= j, \ i \geq n - s\\
    0, & \text{ otherwise.}
    \end{cases}
\end{align}
The free parameters are chosen such that the boundary has accuracy $p/2$ for an order $p$ stencil, the matrix $S$ is negative semi-definite and the leading order truncation errors are minimised. 
This procedure was used in \cite{}, but done so heuristically. 
We pose this as a convex minimisation problem and solve it with the ADMM framework. 
From now we will take $\mathcal{P} \sim \R^{m} $ to be the parameter space for our free variables. 
We will use the variables $\theta_{i,j}$ to represent these free parameters in our optimisation, remembering we still need to solve for the $q_{i,j}$ and $h_i$ variables.

For the following passage, the set $\mathscr{C}_{\varepsilon_1}$ parameterises how negative semi-definite we wish to make the matrix $S$, the set $\mathscr{B}$ can be used to specify already known entries of $S$, and the set $\mathscr{A}_{\varepsilon_2}$ describes the accuracy constraints with additional constraints to enforce the matrix $H$ to be minimally $\varepsilon_2$ positive semi-definite. 

\begin{prop}
Let $\mathcal{S}^{n \times n}$ be the space of symmetric $n \times n$ matrices. 
Define the linear operator $S: \mathcal{P} \mapsto \mathcal{S}^{n \times n}$ through
\begin{align}
    (S(\theta))_{i,j} &\coloneqq 
    \begin{cases}
(\theta_{i,j} + \theta_{j,i})/2,& \text{ if } i,j \leq s\\
(\theta_{s+1-j,s+1-i}+\theta_{s+1-i,s+1-j})/2 ,& \text{ if } i,j \geq n - s\\
 s_k   ,& \text{ else if } i = j - k \\
0, & \text{ otherwise,}
\end{cases};
\end{align}
for $k \in \{ - \max\{|r_1|,|r_2| \}, \max\{|r_1|,|r_2| \} \}$ where $s \in \R^{\max\{r_1,r_2 \} }$ is the interior stencil of the matrix $S = (Q_+^T + Q_+ )/2$. 
Further define the sets 
\begin{align}
     \mathscr{C}_{\varepsilon_1} &\coloneqq \{ x \in \mathcal{S}^{n \times n} \ | \ - \l u,  xv \r \geq \varepsilon_1, \ \forall u,v \in \R^n  \}; \\
    \mathscr{B} &\coloneqq \{ x \in \mathcal{S}^{n \times n} \ | \ x_{i,j} = s_{i,j} \  \forall (i,j) \in b \subsetneq \{1, \cdots , n \}^2 \};
\end{align}
and ${A}_1 : \mathcal{P} \mapsto \R^{n_1}$, and ${A}_2 : \mathcal{P} \mapsto \R^{n_2}$ be linear operators and 
\begin{align}
    \mathscr{A}_{\varepsilon_2} \coloneqq \{ \theta  \in \mathcal{P} \ | \ A_1(\theta) = 0, \ A_2(\theta) \geq \varepsilon_2  \}.
\end{align}
Then the sets $\mathscr{A}_{\varepsilon_2}, \mathscr{B},\mathscr{C}_{\varepsilon_1}$ are closed and convex.
Define the functions 
\begin{align}
    \delta_E (x) \coloneqq
    \begin{cases}
    \infty & \text{if } x \in E \\
    0 & \text{otherwise},
    \end{cases}
\end{align}
then the functions $\delta_{\mathscr{A}_{\varepsilon_2}}, \delta_{\mathscr{B}} (S (\cdot)),\delta_{\mathscr{C}_{\varepsilon_1}}(S (\cdot))$ are convex. 
Let $c: \mathcal{P} \mapsto \R_+$ be a strongly convex function, so in particular, the function $C_{\varepsilon_1, \varepsilon_2} : \mathcal{P} \mapsto \R_+$ given through 
\begin{align}
    {C}_{\varepsilon_1, \varepsilon_2}(\theta) \coloneqq \delta_{\mathscr{A}_{\varepsilon_2}} (\theta) +  \delta_{\mathscr{B} } (S(\theta) ) + \delta_{\mathscr{C}_{\varepsilon_1}}(S(\theta)) + c(\theta),
\end{align}
is strongly convex.
\end{prop}
\begin{proof}
First, the operator $S$ is convex as it is linear.
The set $C_{\varepsilon_1}$ is convex as for each element $x \in \mathscr{C}_{\varepsilon_1}$ we have 
\begin{align}
     - \l u,  (tx_1 + (1-t) x_2) v \r = - t\l u,  x_1  v \r - (1-t) \l u,   x_2 v \r \geq \varepsilon_1.
\end{align}
Similarly, the sets $\mathscr{B}$ and $\mathscr{A}_{\varepsilon_2}$ are linear spaces, making them convex. 
The function $\delta_E$ is convex whenever the set $E$ is closed and convex, which is true for all the sets defined as for convergent sequences in any of the sets defined the limit is also in this set as no strict inequalities are used.
Finally, the sum of convex and strongly convex functions is strongly convex completing the claim. 

\end{proof}

Unfortunately, the indicator function is not smooth, and so we have posed a non-smooth convex optimisation problem. 
We will use the ADMM projection method to solve this. 

\begin{align}
    \theta^* = \argmin_{\theta \in \mathcal{P}, \ \theta', \theta'' \in S(\mathcal{P}) } \{ \delta_{\mathscr{A}_{\varepsilon_2}} (\theta) +  \delta_{\mathscr{B} } (\theta'' ) + \delta_{ \mathscr{C}_{\varepsilon_1}}(\theta') + c(\theta) \}
\end{align}
subject to 
\begin{align}
    \theta' = S(\theta), && \theta'' = S(\theta). 
\end{align}

Let $\| \cdot \|_F$ be the standard Frobenius norm, then the augmented Lagrangian for this problem is
\begin{align}
   \mathcal{L}_t(\theta, \theta' , \theta ''; \lambda_1, \lambda_2) =  C_{\varepsilon_1, \varepsilon_2}(\theta, \theta' , \theta '') &- \l \lambda_1 , S(\theta) - \theta' \r_F   - \l \lambda_2 , S(\theta) - \theta'' \r_F \\
   &+ (t/2)\| S(\theta) - \theta '  \|_F^2 + (t/2)\| S(\theta) - \theta ''  \|_F^2 .
\end{align}

We have the three sub-problems to solve:
\begin{align}
    \argmin_{\theta '' } \mathcal{L}_t(\theta, \theta' , \theta ''; \lambda_1, \lambda_2), \\
    \argmin_{\theta ' } \mathcal{L}_t(\theta, \theta' , \theta ''; \lambda_1, \lambda_2), \\
    \argmin_{\theta  } \mathcal{L}_t(\theta, \theta' , \theta ''; \lambda_1, \lambda_2) .
\end{align}

The last can be efficiently computed with a least squares solver as a convex least squares minimisation with convex constraints.
For our problem, the accuracy constraints are polynomials in the $q_{i,j}$ and $h_i$ variables. 
These equations are of the form 
\begin{align}
    \sum_{i} c_i q_{i,j} h_i = 0,
\end{align}
for all stencils of accuracy $p > 1$. 
We may eliminate the $h_i$ values from these equations as we know they are strictly positive and form a linear system, call this system $\mathbf{C} (\theta) = 0$.
To ensure regularity on our solutions we consider the ridge-regression 
\begin{align}
    \theta^* = \argmin_{\theta \in \mathcal{P}} \| \mathbf{C} (\theta) \|_2^2 + \lambda_c \| \theta \|_2^2 = \argmin_{\theta \in \mathcal{P}} c(\theta),
\end{align}
for a smoothing parameter $\lambda_c>0$. 
The other two can be solved through using projective methods. 
We firstly require the following results. 
\begin{lemma}
Let $P_{\mathscr{C}_{\varepsilon_1}} : \mathcal{S} \mapsto \mathscr{C}_{\varepsilon_1}$ be the least squares projection operator onto $\mathscr{C}_{\varepsilon_1}$, then
\begin{align}
     \argmin_{\theta '' } \mathcal{L}_t(\theta, \theta' , \theta ''; \lambda_1, \lambda_2) = P_{\mathscr{C}_{\varepsilon_1}}( S(\theta_k) - (\lambda_1)_{k}/t),
\end{align}
where $t>0$.
Similarly,
\begin{align}
     \argmin_{\theta ' } \mathcal{L}_t(\theta, \theta' , \theta ''; \lambda_1, \lambda_2) = P_{\mathscr{B} }( S(\theta_k) - (\lambda_1)_{k}/t),
\end{align}
for some $t>0$. 
\end{lemma}
\begin{proof}
The minimisation of the Lagrangian is equivalent to the proximal 
\begin{align}
    \argmin_{\theta''}   \{ \delta_{\mathscr{A}_{\varepsilon_2}} (\theta'')+  \l \lambda_2 ,\theta'' \r_F + (t/2)\| S(\theta) - \theta ''  \|_F^2 \}.
\end{align}
Restricting the argument values, shifting by a constant and completing the square yields
\begin{align}
    \argmin_{\theta'' \in \mathscr{A}_{\varepsilon_2} } \| S(\theta) - \lambda_2/t - \theta '' \|,
\end{align}
which is the least squares projection operator. 
For all $x \in \mathcal{S}$ we have the decomposition $x = Q^T x_dQ$ where $Q$ is orthonormal and $x_d$ is a diagonal matrix. 
From here it is straight-forward to find that 
\begin{align}
    \argmin_{\theta'' \in \mathscr{A}_{\varepsilon_2} } \| S(\theta) - \lambda_2/t - \theta '' \| = Q^T (x_d)_{\varepsilon_2} Q,
\end{align}
where the entries of $(x_d)_{\varepsilon_2}$ are the minimum of $-\varepsilon_2$ and the diagonal entries of $x_d$.
An analogous argument holds for the second case with the appropriate projection operator simply being replacing the entries of $S(\theta) - \lambda_1/t - \theta '$ with those specified in $\mathscr{B}$. 
\end{proof}

This leads to the 3-step ADMM algorithm, which has guaranteed convergence due to the co-efficient matrices involved having one pair-wise orthogonal condition \cite{chen2013convergence}. 
Now the update rule for the ADMM method is:
\begin{align}
    \theta_{k+1}'' &= P_{\mathscr{C}_{\varepsilon_1}}( S(\theta_k) - (\lambda_1)_{k}/t) \\
    \theta_{k+1}' &= P_{\mathscr{B} }( S(\theta_k) - (\lambda_2)_{k}/t) \\
    \theta_{k+1} &= \argmin_{\theta \in \mathscr{A}_{\varepsilon_2} } \left\{ \frac{2}{t} c(\theta) +\| S(\theta) - \theta_{k+1} ' - (\lambda_1)_{k}/t \|_F^2 + \| S(\theta) - \theta_{k+1} '' - (\lambda_2)_{k}/t \|_F^2  \right\}\\
    (\lambda_1)_{k+1} &= (\lambda_1)_{k} - t (S(\theta_{k+1}) - \theta_{k+1} ' )\\
     (\lambda_1)_{k+1} &= (\lambda_2)_{k} - t (S(\theta_{k+1}) - \theta_{k+1} '' ).
\end{align}

The results for orders $4,5,6,7$ are given in the appendix.
These stencils can be found in rational arithmetic, but are presented for readability in floating point. 

\section{Semi-discrete analysis in 1D}

\subsection{Refinement Error Bounds}

In this section we consider how the dispersion error acts under grid-refinement. 
We see to gain asymptotic results, we only require the dispersion approximate $\bom$ to be well behaved around zero. 
To see this, we consider the behaviour of $T_h\bom$ where $T_h$ is an operator that maps $\bom$ to its refinement by a factor of $h$ and compositions of these maps.
More rigorously, let
\begin{align}
    \mathscr{h} \in \ell_{< 1, \infty} \coloneqq \{ \mathscr{h} \in \ell_{\infty} \ | \ \mathscr{h}_i < 1 \}.
\end{align}

For each $\mathscr{h} \in \ell_{< 1, \infty}$ we associate the sequence of spaces 
\begin{align}
    X_{\mathscr{h}} = (X_0, X_1, \dots),
\end{align}
where $X_0 = [-\pi, \pi]$ and $X_k \coloneqq [-\prod_{j = 1}^{k} (\frac{1}{h_{j}}) \pi , \prod_{j = 1}^{k} (\frac{1}{h_{j}}) \pi]$. 
Define the grid refinement operator $T_{\mathscr{h},j} : L^2(X_{j-1}) \mapsto  L^2(X_{j})$ through 
\begin{align}
    T_{\mathscr{h},j} (g) =  \frac{1}{h_j} g (\ \cdot \ h_j).
\end{align}
The dispersion relations we aim to approximate are left invariant under this operator. 
Further, this operator represents grid-refinement by a factor of $0<h_j<1$. 
We must consider sequence of spaces to properly define the point-wise limit as the refinement factor approaches zero, as the linear opertor $T_{h,l}$ becomes unbounded in this case. 
\begin{prop}
For any $\mathscr{h} \in \ell_{< 1, \infty}$, the associated sequence of operators $(T_{\mathscr{h},1}, T_{\mathscr{h},2}, \dots)$ are bounded linear operators. 
Furthermore, the functions of the form $g^2 = |c \ \cdot \ |^2  $ obey 
\begin{align}
    T_{\mathscr{h},j}g \Big|_{X_{j-1}} = g \Big|_{X_{j}},
\end{align}
for all $c \in \R$. 
\end{prop}

\begin{proof}
First we have the bound 
\begin{align}
    \|T_{\mathscr{h},j}g\|_2 \leq \frac{1}{h_i^2} \| g \|_2,
\end{align}
proving that $T_{\mathscr{h},j} : L^2(X_{j-1}) \mapsto  L^2(X_{j})$ and is bounded. 
Further, it is straightforward to check $T_{\mathscr{h},j}$ is linear. 
Now for $x > 0$ we have 
\begin{align}
    T_{\mathscr{h},j} ( x|_{X_{j-1}} \mapsto c x  |_{X_{j-1}}) = x|_{X_{j-1}} \mapsto cx |_{X_{j}}. 
\end{align}
An identical argument holds for $x < 0$. 
\end{proof}

Now we main prove our asymptotic result. 

\begin{theorem}
Let $\omega, \bom \in L^{\infty}(X_0)$, $\mathscr{h} \in \ell_{<1, \infty}$ be given, if 
\begin{align}
    | \omega(k') - \bom(k') | = O({k'}^p)
\end{align}
for $k'$ in a small enough neighbourhood around $k = 0$ then for any $k \in [- \pi, \pi]$ there is $m$ large enough so
\begin{align}
    |  T_{\mathscr{h},l_0} \cdots T_{\mathscr{h},l_m} ( \omega(k) - \bom (k) ) | < O(\prod_{l = 0}^m \left(h_l \omega)^p\right). 
\end{align}
\end{theorem}

\begin{proof}
Using a $p-1$ order Taylor series in $k$ around zero gives the linear approximation. 
Due to the invariance of $\omega$ under $T_{\mathscr{h},l}$ for all $l$, there exists $m$ sufficiently large so that any initial $k$ can be mapped to a neighbourhood of zero through refinement. 
\end{proof}

\begin{cor}
Let $\bom$ that is a $p^{th}$ order approximant to $\omega$ around zero, then 
\begin{align}
    T_{\mathscr{h},l_0} \cdots T_{\mathscr{h},l_m} (\bom ) \xrightarrow[p.w]{m \rightarrow \infty} \omega|_{\R}.
\end{align}
\end{cor}

This shows that higher order methods give better dispersion error bounds around zero for less refined grids.
We have not analysed the high frequency components on a grid however, we do this in the next section.

\subsection{Uniform Error Bounds}

In this section we derive key properties of the maximal dispersion error.
First we show that the defined quantity is invariant under grid refinement.
Next we show that for any wave mode on a given mesh, the maximal dispersion error bounds the dispersion error of any of the present wave modes. 
Last, we show that the maximal dispersion error of the $\pi$-mode bounds all of the dispersion error quantities.

\begin{prop}
The maximal dispersion error is invariant under the action of $T_{\mathscr{h},{l}}$, that is
\begin{align}
    \varepsilon_{\infty} (T_{\mathscr{h},{l}} \bom) =  \varepsilon_{\infty} (\bom) . 
\end{align}
\end{prop}
\begin{proof}
Directly 
\begin{align}
    \varepsilon_{\infty} (T_{\mathscr{h},{l}} \bom) = \max_{k \in [-\pi, \pi]} \frac{\varepsilon_d ( T_{\mathscr{h},{l}} \bom, k) }{ T_{\mathscr{h},{l}} \omega(k) } = \max_{k \in [-\pi/h, \pi/h]} \frac{\varepsilon_d ( \bom, h_l k) }{  \omega (h_lk) } = \varepsilon_{\infty} (\bom).
\end{align}
\end{proof}

For traditional central stencil operators, it is always true that 
\begin{align}
    \varepsilon_d(\bom,k') \leq \varepsilon_d(\bom, k),
\end{align}
for $k' \leq k$.
For our new operators, this is not the case and so we must introduce 
\begin{align}
    \epsilon_d \coloneqq \max_{ |k'| \leq |k| } \varepsilon_d(\bom, k'),
\end{align}
then $\epsilon_d$ is monotonic in $k$.
Note that $\sup_{k} \epsilon_d(\bom, k) \leq \varepsilon_{\infty} (\bom)$.
We may now state how our $\alpha$-DRP schemes control the dispersion error uniformly with mesh refinement. 
\begin{theorem}
Let $\mathscr{h} \in \ell$, $T_{\mathscr{h},l}$ be the associated sequence of refinement operators, then
\begin{align}
    \varepsilon_d (T_{\mathscr{h},l} (\bom), k) \leq \epsilon_d (T_{\mathscr{h},l}(\bom), k) \leq \varepsilon_\infty (\bom) \leq \alpha,
\end{align}
independent of $k$. 
\end{theorem}
Last, we provide the quantitative means to compare the uniform bounds on our new operators to grid-refinement of traditional schemes.
For each of our new operators, we find 
\begin{align}
    h^*(k) = \arg \max_{h \in (0,1]} \epsilon_d(T_h(\bom_{trad} ), k) \leq \delta_k,
\end{align}
where $\delta_k = \epsilon_d(\bom, k)$ for one of our newly derived operators $\bom$. 
The value $ 1/h^*(k)$ represents the resolution refinement needed for the traditional operators to match the dispersion error or our newly derived operators for the wavenumber $k$. 
In three spatial dimensions for a time dependant problem, the computational effort to achieve the same accuracy scales to be $\frac{1}{(h^*(k) )^{3+1} }$ for this level of accuracy. 

\section{Computations}

\subsection{Dispersion Relation Comparisons}
In this section we compare the dispersion relations of our newly derived operators and traditional finite difference schemes. 
We find that our dispersion error is far smaller than traditional schemes for high-frequencies.
Further, the phase velocity for the computed with our operators is far closer to the continuous setting than standard schemes. 
Last we see that our operators do not support spurious wave modes, unlike their traditional counterparts.

\begin{figure}[H]
\centering
  \noindent
  \makebox[\textwidth]{\includegraphics[width =1.2\textwidth]{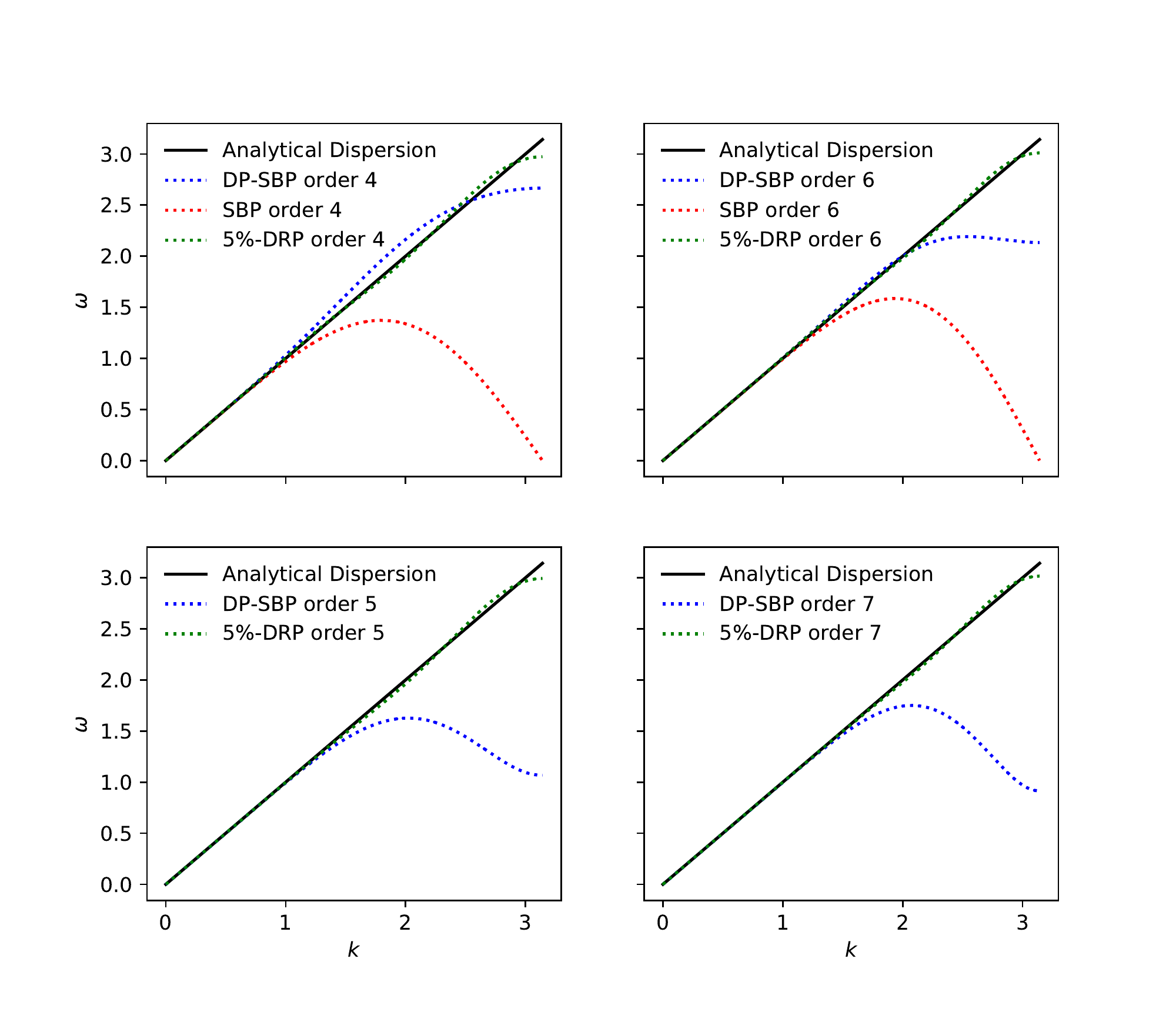}}%
  \caption{Dispersion relations plotted in the upper-quadrant. Here we compare the current standard (SBP), the dual-paring operators (DP-SBP) and our newly derived operators.}
\end{figure}

\begin{table}[H]
    \centering
    \begin{tabular}{c||c|c|c|c}
     & order 4 &  order 5 &  order 6 & order 7 \\
     \hline
     SBP  &  $64.35 \%$ & -   & $58.9 \%$ & - \\
     DP   &  $7.69 \%$ & $43.86 \%$ & $17.54 \%$ & $44.16 \%$ \\
     DRP  &  $1.91 \%$ & $1.72 \%$ &  $1.36 \%$ & $1.28 \%$ \\
    \end{tabular}
    \caption{$L^2$ relative errors for various dispersion relations. }
    \label{tab:drpl2}
\end{table}

Given the dispersion relation, we may compute the phase velocity through 
\begin{align}
    v_p(k) = \frac{\bom(k) }{k},
\end{align}
where in our normalised case the analytic solution is $v_p = 1$. 
In Figure \ref{} we plot the phase velocity as a function of $k$ and compare them to the standard stencils.

\begin{figure}[H]
\centering
  \noindent
     \makebox[\textwidth]{\includegraphics[width =1.2\textwidth]{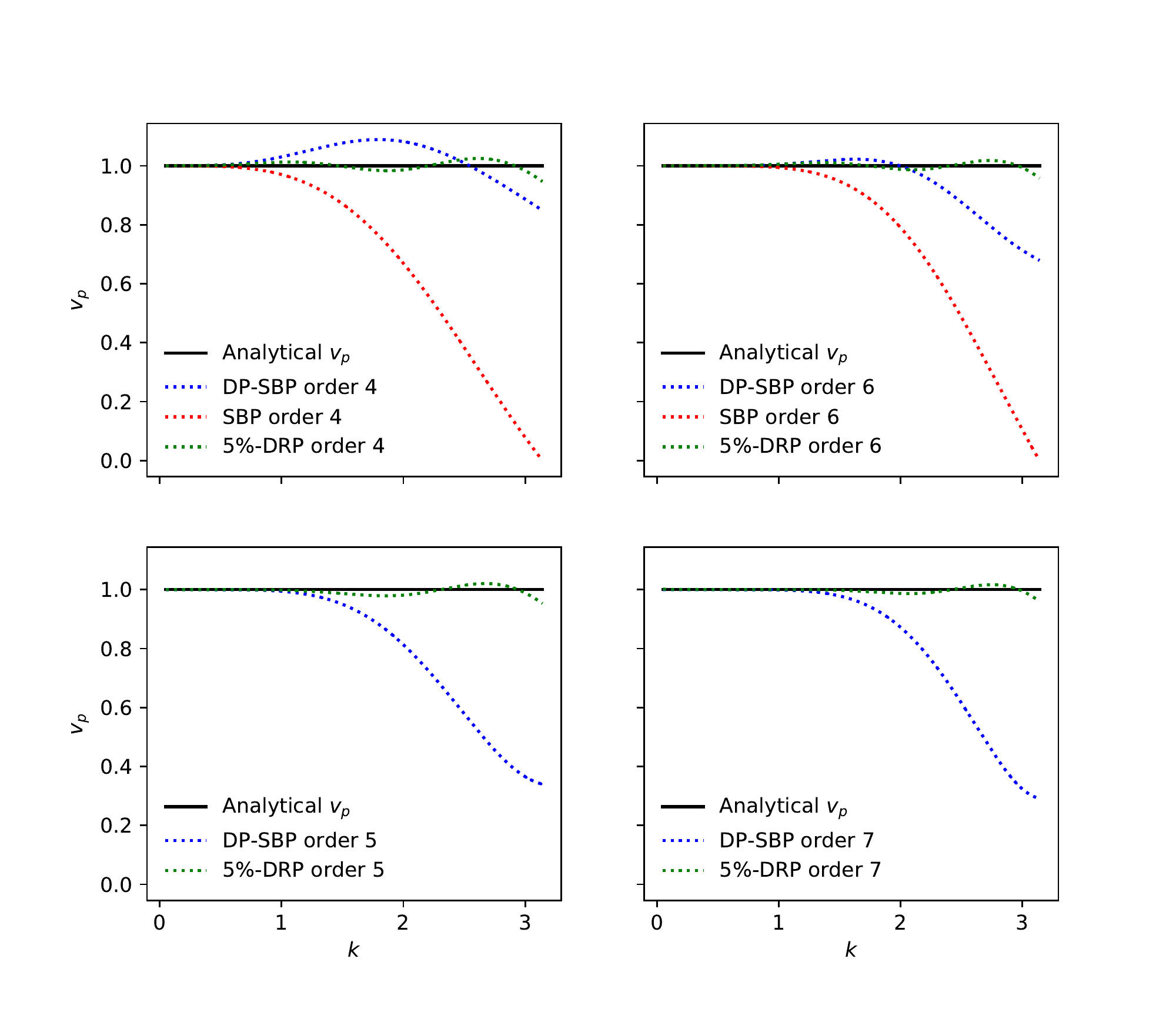}}%
  \caption{Phase Velocities plotted in the upper-quadrant. Here we compare the current standard (SBP), the dual-paring operators (DP-SBP) and our newly derived operators.}
\end{figure}

\begin{table}[H]
    \centering
    \begin{tabular}{c||c|c|c|c}
     & order 4 &  order 5 &  order 6 & order 7 \\
     \hline
     SBP  &  $42.83 \%$ & -   & $38.13 \%$ & - \\
     DP   &  $6.01 \%$ & $28.87 \%$ & $11.05 \%$ & $28.47 \%$ \\
     DRP  &  $1.38 \%$ & $2.53 \%$ &  $0.982 \%$ & $3.47\%$ \\
    \end{tabular}
    \caption{$L^2$ relative errors for various phase velocities. }
    \label{tab:my_label}
\end{table}

\begin{figure}[H]
\centering
  \noindent
  \makebox[\textwidth]{\includegraphics[width =1.2\textwidth]{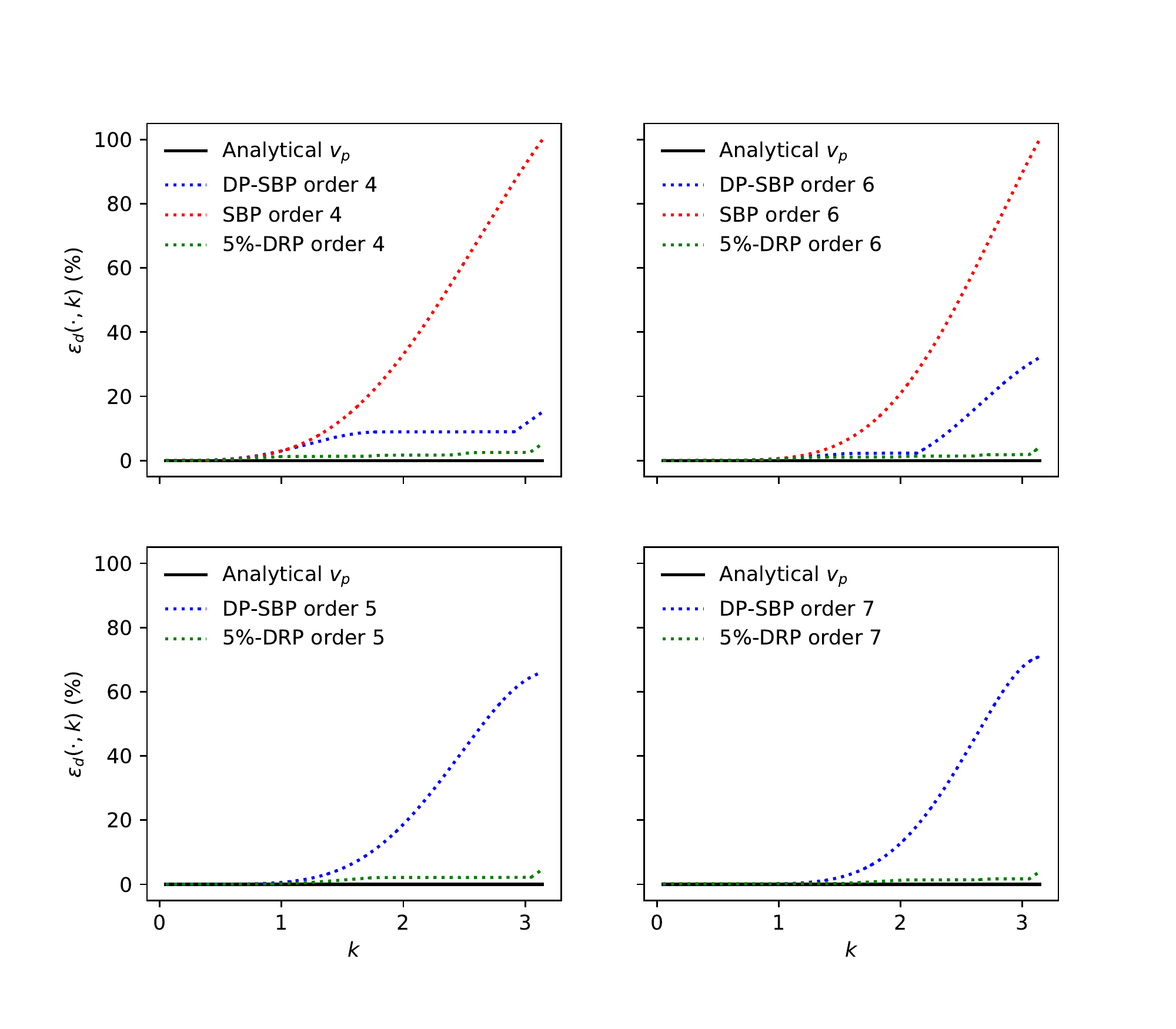}}%
  \caption{Error plot for the values $\epsilon_d(\cdot, k)$ for various operators. Here we compare the current standard (SBP), the dual-paring operators (DP-SBP) and our newly derived operators.}
\end{figure}

We may compare the dispersion errors of the various schemes, this is the value $\epsilon_d$.
All of the new schemes are bounded by $\alpha = 5 \%$.
The traditional schemes approach $100 \%$ toward the high-frequency end of the spectrum, and the dual-pairing operators gain a considerable amount of error in the high-frequency region.

Last, we may compare the $h^*$ values which represent how much refinement is needed for the traditional operators to gain the same dispersion error bounds as our new schemes. 
Recall this was:
\begin{align}
    h^*(k) = \arg \max_{h \in (0,1]} \epsilon_d(T_h(\bom_{trad} ), k) .
\end{align}
We give a desired tolerance $\delta_d$ which is the acceptable dispersion error for a simulation. 
We then find what is the necessary refinement factor $h^*$ needed to achieve this accuracy for the various schemes. 
The value $1/h^*$ represents how much extra computational resources is needed to achieve this tolerance bound, with the scaling $(1/h^*)^{3+1}$ for the resources needed for a three dimensional time dependent problem.\\

We may also calculate the analytic gradients for each $\bom$ and check if the scheme supports Spurious Wave Modes (SWM). 
This occurs if the gradient of the dispersion relation has the incorrect sign compared to the analytic dispersion relation, and supports waves travelling in the wrong direction. 

\begin{table}
\centering
\makebox[\textwidth]{
\begin{tabular}{c|l||rrr|rrr|rrr|c}
& & \multicolumn{3}{c}{$\delta_d = 0.05$ } & \multicolumn{3}{c}{$\delta_d = 0.025$ } & \multicolumn{3}{c}{$\delta_d = 0.015$ } &  \\
\hline
order& scheme &$h^*$ & $1/h^*$ & $(1/h^*)^4$ & $h^*$ & $1/h^*$ & $(1/h^*)^4$ & $h^*$ & $1/h^*$ & $(1/h^*)^4$ & SWM \\
\hline
4 &DRP & 1.00 & 1.00 &  1.01 & 0.97 & 1.03 &   1.14 & 0.56 & 1.79 &  10.21 & N\\
 &DP-SBP &0.38 & 2.61 & 46.16 & 0.30 & 3.34 & 124.33 & 0.26 & 3.92 & 235.88 & Y\\
 &SBP &0.37 & 2.73 & 55.17 & 0.30 & 3.28 & 115.73 & 0.27 & 3.75 & 198.28 & Y\\
 \hline
5 &DRP &1.00 & 1.00 &  1.00 & 0.97 & 1.03 &   1.11 & 0.49 & 2.03 &  17.14 & N\\
 &DP-SBP &0.48 & 2.09 & 19.13 & 0.42 & 2.39 &  32.87 & 0.38 & 2.64 &  48.35 & Y\\
 \hline
6 &DRP &1.00 & 1.00 &  1.00 & 0.98 & 1.02 &   1.08 & 0.83 & 1.20 &   2.09 & N \\
 &DP-SBP &0.72 & 1.39 &  3.77 & 0.68 & 1.46 &   4.60 & 0.42 & 2.39 &  32.57 & Y \\
 &SBP &0.47 & 2.11 & 19.94 & 0.41 & 2.41 &  33.82 & 0.38 & 2.65 &  49.37 & Y \\
  \hline
7 &DRP &1.00 & 1.00 &  1.00 & 0.98 & 1.02 &   1.07 & 0.85 & 1.18 &   1.93 & N \\
 &DP-SBP &0.54 & 1.84 & 11.41 & 0.49 & 2.05 &  17.56 & 0.45 & 2.21 &  23.84 & Y \\
\hline
\end{tabular}
}
\caption{The refinement factors and associated increases in computational effort needed for various schemes for a desired error tolerance. Also shown is if the given scheme supports Spurious Wave Modes (SWM).}
\end{table}

\newpage
\subsection{First Order Hyperbolic System}

Our computational domain in the interval $[0,8]$ and we set the wave-speed to be one. 
We impose reflecting boundary conditions at the endpoints $x = 0$ and $x = 8$. 

At $t = 4s$, the wave has hit the boundary, and on the interior of the domain, the analytic solution is identically zero. 
In Figure \ref{fig:DRPpropint} we plot the computed solutions and see that the SBP and DP sixth order schemes are polluted by dispersion errors. 
Our DRP operator does not seem to have the same numeric artifacts to the same magnitude, wherby we must zoom in by a factor of one thousand to see any form of Gibbs phenomena.\\

At $t=8s$ the wave form has returned to its initial condition. 
We plot this in Figure \ref{fig:DRPv} and see all of the schemes have dispersion errors near the high frequency component. 
However, the dispersion error for the DRP scheme is localised to the `jump' position and has not polluted the rest of the computation. 
We can also see the initial condition plotted in black, with circles to represent the location of the grid points. 
Here we can note how few grid points are used for this high-frequency wave.\\

Next we plot the $L^1$ error from the analytic solution for a short time propagation. 
We do this so that we do not see dispersion errors from the boundary stencils we have not optimised for yet. 

\section{Conclusion}

In this paper we have introduced $\alpha$-DRP schemes for high-frequency linear hyperbolic problems. 
We have given a complete methodology for how to construct these schemes and prove stability in the dual-pairing SBP framework. 
Our schemes have less dispersion error than standard methods for high-frequency wave modes, which we have verfifed in a small text example. 

\begin{figure}[H]
    \centering
    \vspace{-3cm}
    \makebox[\textwidth]{\includegraphics[width = 1.3\textwidth]{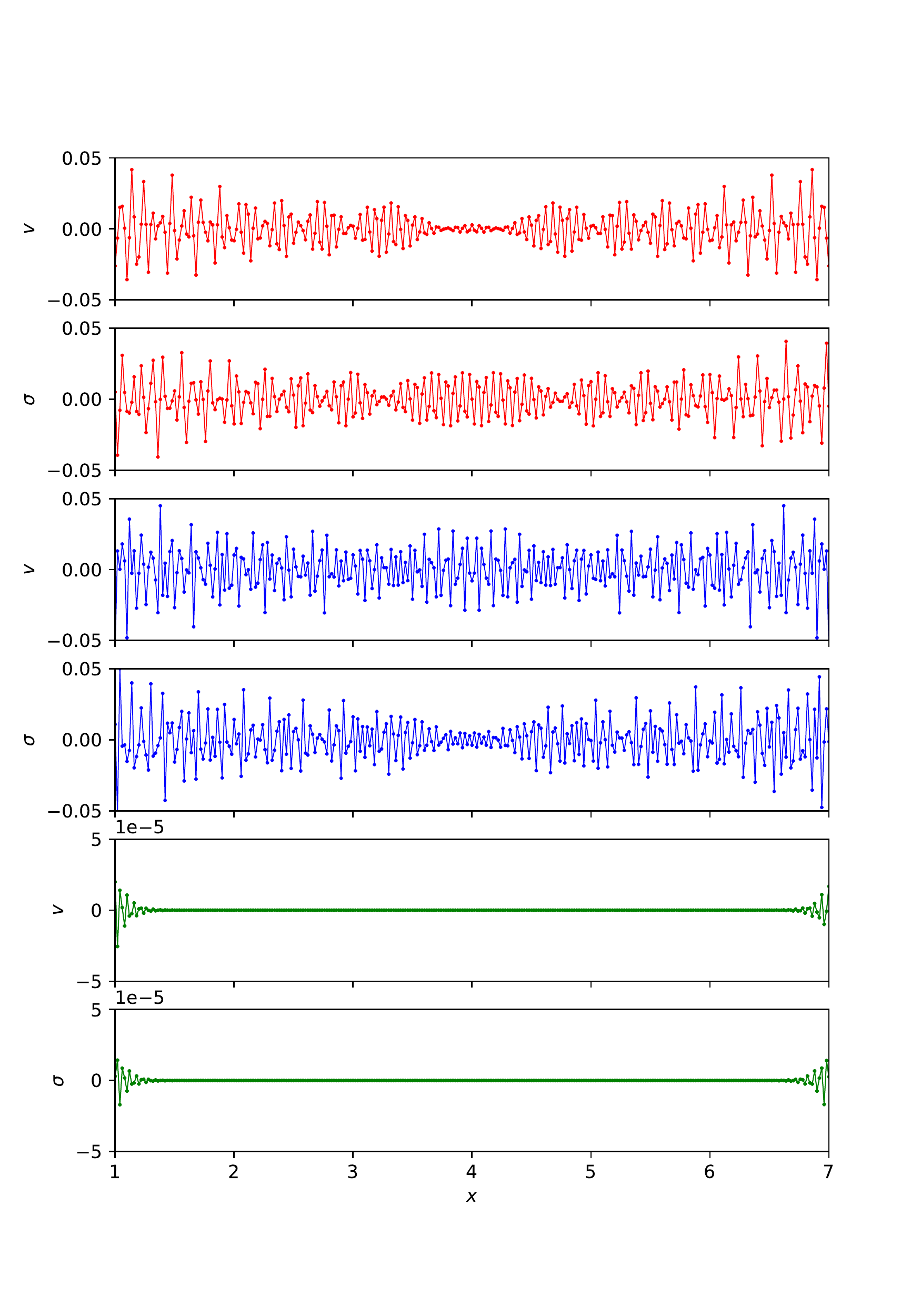}}
    \vspace{-2cm}
    \caption{Internal pollution of solution at time $t = 4s$, the analytic solution should be $v \equiv \sigma \equiv 0$ on $x \in [1,7]$. 
    The sixth order shcemes are: SBP (red), DP (blue), and DRP (green). 
    Notice the DRP scheme is plotted on a scale $1000$ times the resolution of the other two schemes.}
    \label{fig:DRPpropint}
\end{figure}

\begin{figure}[H]
    \centering
    \vspace{-3cm}
    \makebox[\textwidth]{\includegraphics[width = 1.5\textwidth]{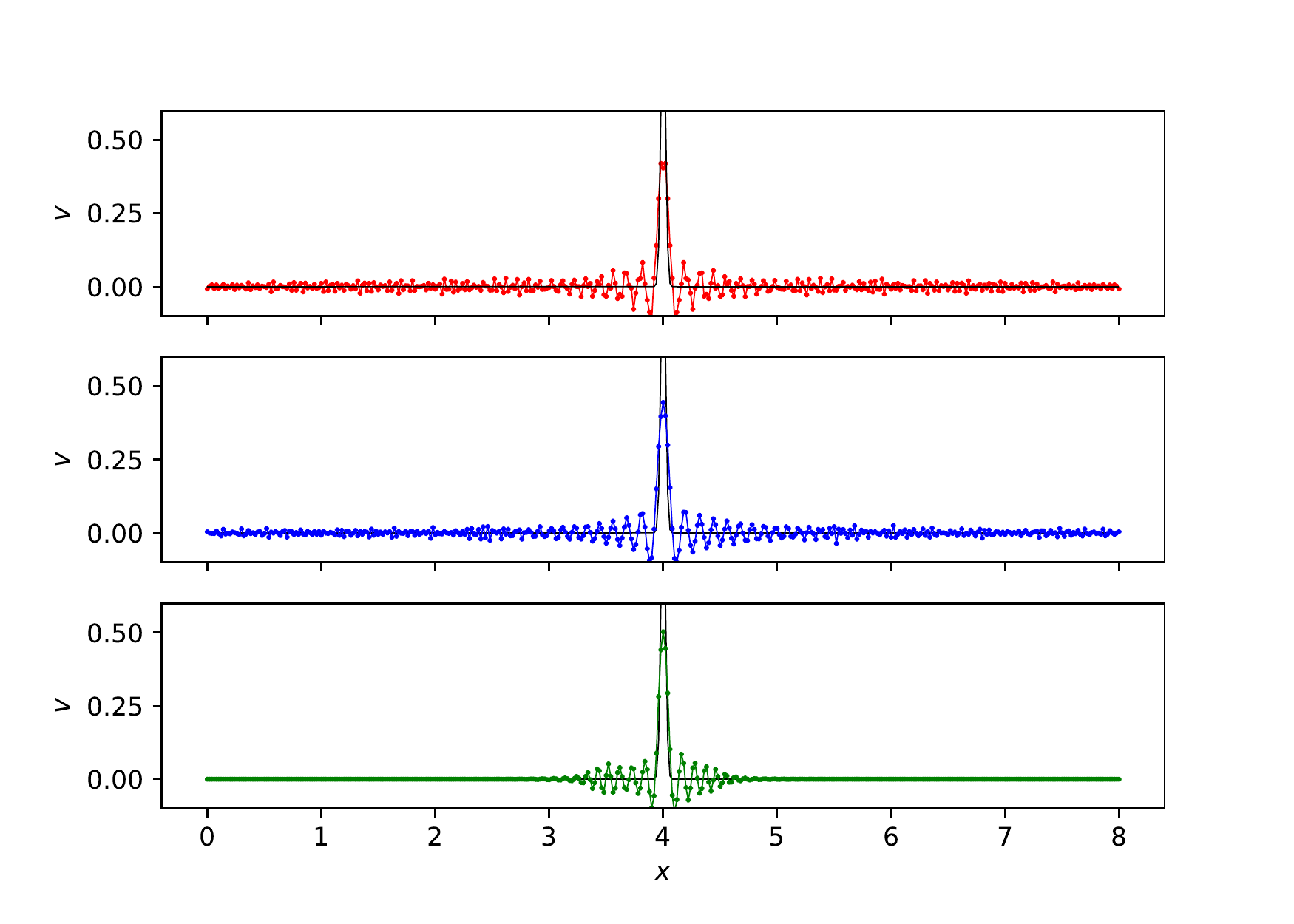}}
    \vspace{-1cm}
    \caption{Sixth order SBP (red) and Fifth order DP (blue) and DRP(green) for the velocity field at $t = 8s$.}
    \makebox[\textwidth]{\includegraphics[width = 1.5\textwidth]{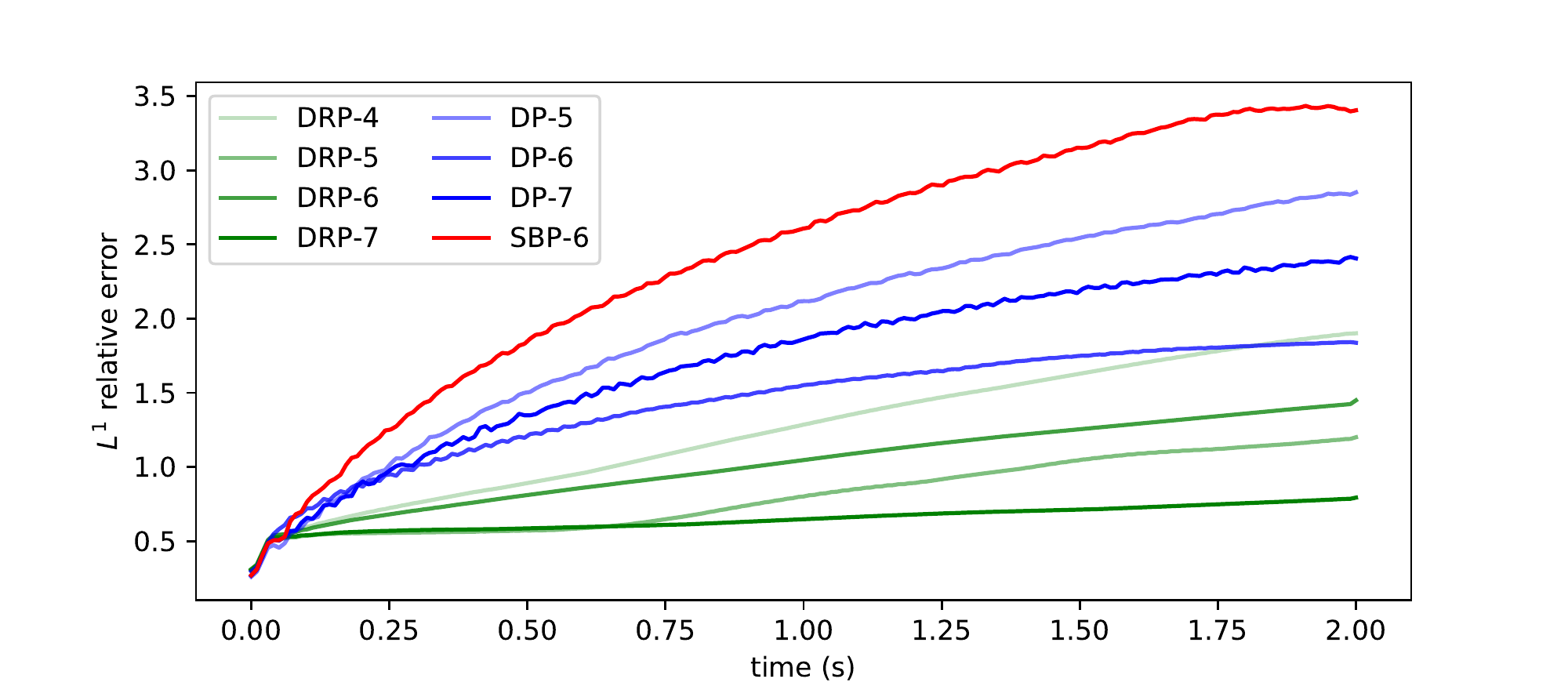}}
    \caption{$L^1$ relative errors for propagation of the $\pi$-mode.}
    \label{fig:DRPv}
\end{figure}

\printbibliography

\appendix

\section{Derivived DRP-Operators}

In this section we give explicit examples of the operators derived and used in our computational experiments. 
All the results were found in exact rational arithmetic. 
For ease of reading, the stencils have been converted floating point numbers for the boundary stencils. 
Due to the SBP relation:
\begin{align}
    (HD_+)^T + (HD_-) = B,
\end{align}
we need only present the values for $Q_+$ and $H$ for each operator. 

\subsection{Interior Stencils}
Here we give the interior stencils for the forward difference operator $Q_+$, and also identically, $D_+$. 
\renewcommand{\arraystretch}{1.5}
\begin{table}[H]
        \centering
       \begin{tabular}{||c||c|c|c|c|c|c|c|c|c|c||}
       \hline
order  &  $w_{-4}$& $w_{-3}$ & $w_{-2}$ & $w_{-1}$ & $w_{0}$ &    $w_{1}$ & $w_{2}$ & $w_{3}$ & $w_{4}$ & $w_{5}$ \\ 
\hline
4    &  - & - & $\frac{-5}{48}$ & $\frac{29}{360}$& $\frac{-401}{360}$ &    $\frac{7}{5}$& $\frac{-187}{720}$ & $\frac{-1}{360}$ & - & -                      \\ \hline
5    &  - & $\frac{13}{525}$ & $\frac{-109}{1050}$ &  $\frac{-17}{175}$& $\frac{-127}{140}$& $ \frac{31}{21}$& $\frac{-167}{350}$& $\frac{47}{525}$  & $\frac{-11}{2100}$ & -                       \\ \hline
6   &  $\frac{-1}{168}$ & $\frac{149}{3150}$ & $\frac{-199}{1575}$ & $\frac{-8}{75}$ & $ \frac{-8}{9}$ & $\frac{67}{45}$ & $\frac{-37}{75}$ & $\frac{124}{1575}$ & $\frac{139}{12600}$  & $\frac{-1}{210}$            \\ \hline
7     &  $\frac{-43}{7056}$ & $\frac{4859}{117600}$ &  $\frac{-107}{1225}$ & $\frac{-841}{4200}$ & $\frac{-1111}{1400}$ & $\frac{119}{80}$ & $\frac{-617}{1050}$ &    $\frac{5113}{29400}$ & $\frac{-587}{19600}$    & $\frac{737}{352800}$  \\ \hline
\end{tabular}
\renewcommand{\arraystretch}{1}

\caption{Upwind $5\%$-DRP forward SBP finite difference coefficients for the interior stencils with interior order of accuracy $ 4, 5, 6, 7$. }
\label{tab:drp_int}
\end{table}

\subsection{Boundary Stencils}

Here we give the co-efficients for the boundary stencils of the $Q_+$ matrices for various orders. 
We have made sure to enforce that $H>0$ and that the smallest eigenvalue of $H$ is approximately $0.3$ which matches the traditional case so that the same time integration algorithms may be used.

\subsubsection{Order 4}

\paragraph{H Stencil}



\begin{align}
 \begin {tabular}{cccc}  
     $h_1$ &$h_2$ &$h_3$ &$h_4$\\
     \hline 
    0.407206,& 1.05763,& 1.07979, &0.955377\end {tabular} 
\end{align}

\paragraph{$Q_+$ Stencil}

\begin{align*}
\{q_{i,j}\} _{i,j = 1}^4 = 
\begin{tabular}{c|cccc}
$i ,\ j$ &1&2&3&4\\
\hline
 1& - 0.0371647& 0.690309&- 0.176330&
 0.0231858\\2& - 0.523249&- 0.290858& 1.09105&-
 0.274169\\3&  0.0651977&- 0.431581&- 0.677496&
 1.30638\\ 4& - 0.00478431& 0.0321293&- 0.133060&-
 1.03178
\end{tabular}
\end{align*}

\newpage
\begin{sidewaystable}
\subsubsection{Order 5}

\paragraph{H Stencil}

\begin{align*}
 \begin {tabular}{cccccc}  
     $h_1$ &$h_2$ &$h_3$ &$h_4$&$h_5$&$h_6$\\
     \hline 
    0.318079& 1.38392& 0.632156& 1.24425&
 0.905649& 1.01595\end {tabular} 
\end{align*}

\paragraph{$Q_+$ Stencil}

\begin{align*}
&\{q_{i,j}\} _{i,j = 1}^6 = \\
&\begin {tabular}{c|cccccc} 
$i ,\ j$ &1&2&3&4&5&6\\
\hline
1& - 0.0180780& 0.718170&- 0.182502&-
 0.0120891&- 0.0286837& 0.0231832\\
 2& - 0.673508&-
 0.0323215& 0.704725&- 0.0288675& 0.0564717&- 0.0264995
\\ 3& 0.229024&- 0.746702&- 0.0528860& 0.773477&-
 0.184003&- 0.0136722\\ 4& - 0.0338854& 0.0960408&-
 0.664280&- 0.184729& 0.982364&- 0.279796\\ 5& -
 0.0167986&- 0.0296235& 0.212984&- 0.506050&- 0.643824& 1.37617
\\ 6& 0.0132463&- 0.00556370&- 0.0180414&- 0.0665032&-
 0.103277&- 0.903194
\end{tabular}
\end{align*}
\end{sidewaystable}

\newpage

\begin{sidewaystable}
\subsubsection{Order 6}
\paragraph{H Stencil}
\begin{align*}
    \begin {array}{cccccccc}  
    h_1 & h_2 & h_3 & h_4 & h_5 & h_6 & h_7 & h_8
    \\ 
    \hline
    0.294425& 1.52829& 0.251092& 1.80762
& 0.403131& 1.28497& 0.920654& 1.00981\end {array}
\end{align*}

\paragraph{$Q_+$ Stencil}
\begin{align*}
&\{q_{i,j}\} _{i,j = 1}^8 = \\
&\begin {tabular}{c|cccccccc}
i ,\ j &1&2&3&4&5&6&7&8\\
\hline
1& - 0.00903009& 0.710170&- 0.103201&-
 0.149256&- 0.00889415& 0.0775492&- 0.00762020&- 0.00971753
\\  2& - 0.678377&- 0.0306467& 0.383435& 0.382131&
 0.144309&- 0.260057& 0.0457140& 0.0134917\\  3&
 0.109170&- 0.384446&- 0.0114785& 0.597174&- 0.372833&- 0.00555940&
 0.0897052&- 0.0217330\\ 4& 0.145873&- 0.369601&-
 0.597113&- 0.0223665& 0.509298& 0.736223&- 0.541190& 0.143639
\\ 5& - 0.0410135&- 0.0712000& 0.413473&- 0.467929&-
 0.150629& 0.0494434& 0.470024&- 0.208439\\  6& -
 0.0648459& 0.227413& 0.00266513&- 0.704684& 0.160351&- 0.339735&
 0.894812&- 0.260975\\  7&  0.0512884&- 0.0996057&-
 0.127011& 0.476373&- 0.448256&- 0.0480874&- 0.833357& 1.43699
\\  8& - 0.0130649& 0.0179173& 0.0392310&- 0.111443&
 0.172608&- 0.251125&- 0.0330888&- 0.901590\end {tabular}
\end{align*}
\end{sidewaystable}

\newpage

\begin{sidewaystable}
\subsubsection{Order 7}
\paragraph{H Stencil}
\begin{align*}
    \begin {array}{cccccccc}  
    h_1 & h_2 & h_3 & h_4 & h_5 & h_6 & h_7 & h_8
    \\ 
    \hline
    0.370747& 1.10403& 1.19465& 0.782612
& 0.881724& 1.32497& 0.798505& 1.04278
    \end {array}
\end{align*}

\paragraph{$Q_+$ Stencil}
\begin{align*}
&\{q_{i,j}\} _{i,j = 1}^8 = \\
&\begin {array}{c|cccccccc}
i ,\ j &1&2&3&4&5&6&7&8\\
\hline
 1& - 0.0193844& 0.527413& 0.139639&-
 0.0692289&- 0.100920&- 0.0922254& 0.166400&- 0.0516932
\\  2& - 0.452394&- 0.0840763& 0.279856& 0.238695&
 0.145656& 0.0215770&- 0.266993& 0.117679\\  3& -
 0.136908&- 0.231101&- 0.0701924& 0.139172& 0.200984& 0.229620&-
 0.117601&- 0.0139733\\ 4& 0.0386701&- 0.191651&-
 0.0720692&- 0.0719367& 0.0987782& 0.177441& 0.0955684&- 0.0768901
\\  5& 0.0523818&- 0.0839791&- 0.136216&- 0.0635566&-
 0.175482& 0.248986& 0.193487&- 0.00776149\\  6& 
 0.0827498&- 0.0295226&- 0.211359&- 0.156916& 0.0688144&- 0.530310&
 0.931838&- 0.301346\\  7& - 0.0719838& 0.149688&
 0.00864436&- 0.0685704&- 0.303084& 0.326010&- 1.02574& 1.42660
\\  8& 0.00686850&- 0.0567702& 0.0616977& 0.0523413&
 0.0713473&- 0.416322& 0.0751616&- 0.840256
\end {array}
\end{align*}
\end{sidewaystable}

\begin{sidewaystable}

        \centering
       \begin{tabular}{||ccc||c|c|c|c|c|c|c|c|c|c||}
       \hline
order & $a$ & $b$ &  $w^{(+)}_{-4}$& $w^{(+)}_{-3}$ & $w^{(+)}_{-2}$ & $w^{(+)}_{-1}$ & $w^{(+)}_{0}$ &    $w^{(+)}_{1}$ & $w^{(+)}_{2}$ & $w^{(+)}_{3}$ & $w^{(+)}_{4}$ & $w^{(+)}_{5}$ \\ 
\hline
2   & 0         & 2  &  - & - & - & - & $-3/2$ & $2$ & $-1/2$ & - & - & -                                   \\ \hline
3   & 1         & 2  &  - & - & - & $-1/3$ & $-1/2$ &   $1$ &  $-1/6$ & - & - & -                           \\ \hline
4   & 1         & 3  &  - & - & - & $-1/4$& $-5/6$ &    $3/2$& $-1/2$ & $1/12$ & - & -                      \\ \hline
5   & 2         & 3  &  - & - & $1/20$ &  $-1/2$& $-1/3$& $1$& $-1/4$& $1/30$  & - & -                       \\ \hline
6   & 2         & 4  &  - & - & $1/30$ & $-2/5$ & $ -7/12$ & $4/3$ & $-1/2$ & $2/15$ & $-1/60$  & -             \\ \hline
7   & 3         & 4  &  - & $-1/105$ &  $1/10$ & $-3/5$ & $-1/4$ & $1$ & $-3/10$ &    $1/15$ & $-1/140$    & -  \\ \hline
8   & 3         & 5  &  - & $-1/168$ & $1/14$ & $-1/2$ & $-9/20$ & $5/4$ & $-1/2$ & $1/6$ & $1/28$ & $1/280$          \\ \hline
9   & 4         & 5  & $1/504$& $-1/42$ & $1/7$ & $-2/3$ & $-1/5$ & $1$ & $-1/3$ & $20/21$&  $-1/56$ & $1/630$                    \\ \hline
\end{tabular}
\caption{Upwind forward SBP finite difference coefficients for the interior stencils with interior order of accuracy $2, 3, 4, 5, 6, 7, 8, 9.$ }
\label{tab:upwind}

\end{sidewaystable}

\end{document}